\documentclass[11pt]{article}

\usepackage{amsmath, amssymb, amsthm, color}
\usepackage{hyperref, url}
\usepackage{graphicx}
\usepackage{caption}
\usepackage{subcaption}
\usepackage{ulem}
\usepackage{tikz}
\usepackage{epstopdf}
\usepackage{enumitem}
\usepackage{algorithm}
\usepackage{algorithmicx}
\usepackage{algpseudocode}
\usepackage{mathdots}
\newtheorem{example}{Example}[section]
\newtheorem{remark}{Remark}%
\newtheorem{corollary}{Corollary}[section]
\newtheorem{proposition}{Proposition}
\newtheorem{theorem}{Theorem}
\newtheorem{assumption}{Assumption}

\newtheorem{definition}{Definition}%
\DeclareRobustCommand{\RR}{\mathbb R}
\DeclareRobustCommand{\setS}{\mathbb S}
\DeclareRobustCommand{\setX}{\mathbb X}
\DeclareRobustCommand{\setZ}{\mathbb Z}
\DeclareRobustCommand{\setB}{\mathbb B}
\DeclareRobustCommand{\setA}{\mathbb A}
\DeclareRobustCommand{\setN}{\mathbb N}

\DeclareRobustCommand{\EE}{\mathbb E}
\newtheorem{lemma}{Lemma}[section]
\newcommand{\dist}{\operatorname{dist}}
\newcommand{\dom}{\operatorname{dom}}

\newcommand{\proj}{\operatorname{Proj}}
\newcommand{\interior}{{\mbox{\rm int}}}

\begin{document}

\title{ New Classes of Non-monotone Variational Inequality Problems Solvable via Proximal Gradient on Smooth Gap Functions }


\author{
Lei Zhao\thanks {Institute of Translational Medicine and National Center for Translational Medicine, Shanghai
Jiao Tong University, 200030 Shanghai, China ({\tt l.zhao@sjtu.edu.cn})}
\and
Daoli Zhu\thanks {Antai College of Economics and Management, Shanghai Jiao Tong University, 200030 Shanghai, China
({\tt dlzhu@sjtu.edu.cn}); and School of Data Science, The Chinese University of Hong Kong, Shenzhen, Shenzhen 518172, China}
\and
Shuzhong Zhang\thanks {Department of Industrial and Systems Engineering, University of Minnesota, Minneapolis, MN 55455, USA
({\tt zhangs@umn.edu})}}

\maketitle

\begin{abstract}

In this paper, we study the local linear convergence behavior of proximal-gradient (PG) descent algorithm on a parameterized gap-function reformulation of a smooth but non-monotone variational inequality problem (VIP). The aim is to solve the non-monotone VI problem without assuming the existence of a Minty-type solution. 
We first introduce and study various error bound conditions for the gap functions in relation to the VI model. 
In particular, we show that uniform type error bounds imply level-set type error bounds for composite optimization,  
revealing a key hierarchical structure there. 
As a result, local linear convergence is established under some easy-verifiable conditions induced by level-set error bounds, the gradient Lipschitz condition and a suitable initialization condition. 
Furthermore, for non-monotone affine VIs we present a homotopy continuation scheme that achieves global convergence 
by dynamically tracing a solution path. Our numerical experiments show  the efficacy of the proposed approach, leading to the solutions of a broad class of non-monotone VI problems resulting from the need to compute Nash equilibria, traffic controls, and the GAN models. 

\end{abstract}

\vspace{0.5cm}



\section{Introduction}\label{intro}
Let $F$ be a continuous mapping from $\RR^d$ to $\RR^d$, and $\setX$ be a closed convex subset of $\RR^d$.
In this paper, we consider the following generic variational inequality problem:
\begin{equation}\label{Prob:VIP}
\begin{aligned}
\mbox{{\rm(VIP)}}\qquad&\mbox{Find $x^*\in\setX$ such that}\\
&\langle F(x^*),x-x^*\rangle\geq0,\quad\forall x\in\setX.
\end{aligned}
\end{equation}

In particular, in this paper we focus on the case where $F$ is smooth but possibly non-monotone. Therefore, the above (VIP) is a non-monotone variational inequality problem. We denote $\setX_{\text{\tiny VIP}}^*$ to be the set of all solutions
to (VIP), and we shall assume henceforth $\setX_{\text{\tiny VIP}}^*\neq\emptyset$. 

Problem (VIP) arises naturally in diverse areas such as game theory~\cite{ZZZ24},  traffic assignment~\cite{ND84}, and generative adversarial networks~\cite{GBVVL18}. 
Because of that, 
iterative methods have recently been proposed 
aiming to find solutions for non-monotone variational inequality problems. To guarantee these proposed methods to work as desired,  
often times one needs to make assumptions such as 
a Minty solution or a weak Minty solution exists. As we will get into specifics shortly, this amounts to assuming that 
the set $\setX_{\text{\tiny MVIP}
}^*$ is nonempty for Minty Variational Inequality (MVI), or $\setX_{\text{\tiny WMVIP}}^*\neq\emptyset$ for weak Minty Variational Inequality (weak MVI); see e.g.~Definition~\ref{defi:mvi_wmvi} to be introduced later, which is a basic assumption in~\cite{M20, DDJ21, B22}. However, 
as shown by simple examples such as Examples~\ref{exp:nvip_1} and~\ref{exp:nvip_2}, the conditions of $\setX_{\text{\tiny MVIP}}^*\neq\emptyset$ or $\setX_{\text{\tiny WMVIP}}^*\neq\emptyset$ may not be met. There is an alternative approach to solve non-monotone VI problems, known as the gap function approach. This paper is to study the conditions under which the gap function approach is guaranteed to solve classes of non-monotone VI problems without resorting to any Minty-type conditions. 

Historically, the so-called {\it smooth gap function}\/ was first introduced by Fukushima in~\cite{F92}, where the gap function $g_{\lambda}(x):\RR^d\rightarrow[0,+\infty)$ of (VIP) is defined as:
\begin{equation}\label{eq:SGF_1}
g_{\lambda}(x) := \max_{y\in\setX}\, \left[ \langle F(x),x-y\rangle-\frac{1}{2\lambda}\|x-y\|^2\right] ,
\end{equation}
or equivalently 
\begin{equation}\label{eq:SGF_2}
g_{\lambda}(x) = \max_{y\in\RR^d}\, \left[ \langle F(x),x-y\rangle+\mathcal{I}_{\setX}(x)-\mathcal{I}_{\setX}(y)-\frac{1}{2\lambda}\|x-y\|^2\right] ,
\end{equation}
where $\mathcal{I}_{\setX}(\cdot)$ is the indicator function of set $\setX$, and $\lambda>0$ is a fixed parameter. By resorting to minimizing $g_{\lambda}(x)$, (VIP) can be equivalently transformed to the following smooth optimization problem:
\begin{equation}\label{Prob:OP}
\begin{aligned}
\mbox{{\rm(OP-$g_{\lambda}$)}}\qquad\qquad\qquad\qquad\quad&\min_{x\in\setX} g_{\lambda}(x).\\
\mbox{(Or, equivalently,} 
& \min_{x\in\RR^d}\psi(x)=g_{\lambda}(x)+\mathcal{I}_{\setX}(x))
\end{aligned}
\end{equation}
Note that (OP-$g_{\lambda}$) is a smooth non-convex optimization problem.  
Moreover, $g_{\lambda}(x)\ge 0$ for all $x\in\setX$, and $g_{\lambda}(x)= 0$ 
$\Longleftrightarrow
x\in \setX_{\text{\tiny VIP}}^*$. A natural idea thus arises: Instead of solving (VIP), how about solving simple-looking optimization problem (OP-$g_{\lambda}$)? This is known as the gap function approach to solve (VIP). The only issue with the approach is that we will have to find a {\it global}\/ solution for (OP-$g_{\lambda}$) in order to actually find a solution for (VIP). If we apply a regular first-order method to solve (OP-$g_{\lambda}$), then we can only expect to find a {\it stationary}\/ point of (OP-$g_{\lambda}$). The question then becomes: Under what conditions can we be assured that the solution found is actually a global minimizer? If so, can we further be assured that the local rate of convergence is linear? These are the questions that we set out to answer in the paper. 

Let $\setX_{g_{\lambda}}^*$ and $\overline{\setX}_{g_{\lambda}}$ be solution set and the critical point set of (OP-$g_{\lambda}$) respectively. Since the solution set of (VIP) $\setX_{\text{\tiny VIP}}^*$ is nonempty, we know that  $\setX_{g_{\lambda}}^*=\setX_{\text{\tiny VIP}}^*\neq\emptyset$. 
The stationary point $\bar{x}\in\overline{\setX}_{g_{\lambda}}$ of (OP-$g_{\lambda}$) satisfies the following first order condition: $0\in\partial\psi(x)=\nabla g_{\lambda}(x)+\mathcal{N}_{\setX}(x)$ with $\mathcal{N}_{\setX}(x)$ be the normal cone of convex set $\setX$ at $x$. Let
\begin{equation}\label{eq:yx}
y_{\lambda}(x) := \arg\max_{y\in\setX}\, \left[ \langle F(x),x-y\rangle-\frac{1}{2\lambda}\|x-y\|^2\right] =\proj_{\setX}\left(x-\lambda F(x)\right).
\end{equation}
We can derive that $y_{\lambda}(x)$ is single-valued mapping, and the gap function $g_{\lambda}(x)$ is smooth on $\setX$, and its gradient is 
\begin{equation}\label{eq:grad_g}
\nabla g_{\lambda}(x)=F(x)+\nabla F(x)^{\top}\left[x-y_{\lambda}(x)\right]+\frac{1}{\lambda}\left[y_{\lambda}(x)-x\right].
\end{equation}
Now, consider applying the Proximal Gradient (PG) method
to solve (OP-$g_{\lambda}$) as follows: 
\begin{algorithm}[h]
\caption{Proximal Gradient Method for Solving (OP-$g_{\lambda}$)}
{\bf Initialization:}  $x^0\in\setX$;
\begin{algorithmic}[1]
\For{$k=0,1,\ldots$}
\State Update  $x^{k+1} := \arg\min_{x\in\RR^d}\left[ \langle\nabla g_{\lambda}(x^k),x-x^k\rangle+\mathcal{I}_{\setX}(x)+\frac{1}{2\alpha}\|x-x^k\|^2\right]$.
\EndFor
\end{algorithmic}
\label{alg:PGOPg}
\end{algorithm}
\vspace{-0.5cm}

In this paper, we investigate the PG approach for solving (OP-$g_{\lambda}$) or its parameterized version. Our aim is to identify classes of VI problems on which 
the proposed algorithm is guaranteed to find a solution of the non-monotone VI that we set out to solve. 

\subsection{Non-monotone VI under Minty or weak Minty conditions} 

Non-monotone VI problems and their applications constitute a substantial part of mathematical programming. Recent progresses have been made to solve some classes of non-monotone VI's under the premise that Minty-type conditions are satisfied.  
For example, in~\cite{LRLY21} and~\cite{M20} the existence of a Minty solution (see (2) of Definition~\ref{defi:mvi_wmvi}) is assumed to guarantee the non-monotone VIP to be solvable. Additionally, the extra gradient method (EG) has also been studied recently as another iterative algorithm~\cite{DDJ21} and~\cite{B22}, where the authors studied a class of non-monotone VIP with the existence of a {\it weak Minty} 
solution (see (2) of Definition~\ref{defi:mvi_wmvi}). 
Moreover, in~\cite{DDJ21} the authors commented that the weak MVI
actually generalizes the concept of negative co-monotonicity~\cite{BMW21}. Specifically, for (VIP) these two notions are as follows: 
\begin{definition}\label{defi:mvi_wmvi}
\begin{itemize}
    \item[{\rm(1)}] {\rm(}{\bf Minty variational inequality solution (MVI)}{\rm)} Given an operator $F:\RR^d\rightarrow\RR^d$, there exists a point $x^*\in\setX$ such that $\langle F(x),x-x^*\rangle\geq0,\;\forall x\in\setX$.    
    \item[{\rm(2)}] {\rm(}{\bf Weak Minty variational inequality solution (weak MVI)}{\rm)} Given an operator $F:\RR^d\rightarrow\RR^d$, there exists a point $x^*\in\setX_{\text{\tiny VIP}}^*$ and $\rho>0$ such that $\langle F(x),x-x^*\rangle\geq-\frac{\rho}{2}\|F(x)\|^2,\;\forall x\in\RR^d$.
\end{itemize}
\end{definition}
It is clear that (MVI) or weak (MVI) may not have solutions even though (VIP) has a solution.

\begin{example}\label{exp:nvip_1} 
{\rm Consider 
$F(x)=-x$ and $\setX=\{x\in\RR^d :\: -1\leq x\leq 1\}$. It is easy to check that the solution for the problem is $\setX_{\text{\tiny VIP}}^*=\{-1,1,0\}$, to be denoted as $x_1^*=-1$, $x_2^*=1$, and $x_3^*=0$.
Letting $x_1=1$ and $x_2=0$, we have $\langle F(x_1)-F(x_2),x_1-x_2\rangle=-1<0$. Thus, $F(x)$ is non-monotone. Next, we show that the solution set of MVI for this problem, which is always a subset of $\setX_{\text{\tiny VIP}}^*$, is actually empty. 
To see this, observe $\langle F(x_1),x_1-x_1^*\rangle=-2<0$ for $x_1=1$;
$\langle F(x_2),x_2-x_2^*\rangle=-2<0$ for $x_2=-1$; 
$\langle F(x_3),x_3-x_3^*\rangle=-1<0$ for $x_3=1$.

Let us compute the gap function for this problem. Taking $\lambda=1$, the gap function 
is
\[
g_{1}(x)=-\min_y\langle F(x),y-x\rangle+\frac{1}{2}\|y-x\|^2=\left\{\begin{array}{cl}
\frac{1}{2}x^2,&x\in[-\frac{1}{2},\frac{1}{2}];  \\
-\frac{3}{2}x^2-2x-\frac{1}{2},&x\in[-1,-\frac{1}{2} ) ; \\
-\frac{3}{2}x^2+2x-\frac{1}{2},&x\in(\frac{1}{2},1],
\end{array}\right.
\]
which is weakly convex on $[-1,1]$, the solution set (OP-$g_{1}$) is $\setX_{g_1}^*=\setX_{\text{\tiny VIP}}^*=\{-1,1,0\}$, the critical point set $\overline{\setX}_{g_{1}}=\{-\frac{2}{3}, \frac{2}{3}, 0\}\neq\setX_{\text{\tiny VIP}}^*$. The gradient of the gap function (the right plot in Figure~\ref{fig:01}) is
\[
\nabla g_{1}(x)=\left\{\begin{array}{cl}
x, &x\in[-\frac{1}{2},\frac{1}{2}];  \\
-3x-2, &x\in[-1,-\frac{1}{2}); \\
-3x+2, &x\in(\frac{1}{2},1].
\end{array}\right.
\]

\begin{figure}
\centering
\begin{subfigure}[t]{0.45\textwidth}
\begin{center}
\scriptsize
		\tikzstyle{format}=[rectangle,draw,thin,fill=white]
		\tikzstyle{test}=[diamond,aspect=2,draw,thin]
		\tikzstyle{point}=[coordinate,on grid,]
\begin{tikzpicture}[>=stealth]
 \draw[->](-2,0)--(2,0)node[below]{$x$};
 \draw[->](0,-1.5)--(0,1.5)node[right]{$g_{\lambda}(x)$};
 \node at(-0.2,-0.2){$O$};
 \draw [domain=-0.5:0.5,red] plot(\x,0.5*\x*\x)node[above,xshift=5]{$g_{\lambda}(x)=\frac{1}{2}x^2$};
 \draw [domain=-1:-0.5,blue] plot(\x,-1.5*\x*\x-2*\x-0.5)node[left, yshift=10]{$g_{\lambda}(x)=-\frac{3}{2}x^2-2x-\frac{1}{2}$};
  \draw [domain=0.5:1,blue] plot(\x,-1.5*\x*\x+2*\x-0.5)node[right, xshift=-25, yshift=-13]{$g_{\lambda}(x)=-\frac{3}{2}x^2+2x-\frac{1}{2}$};
\end{tikzpicture}
\end{center}
\end{subfigure}
\begin{subfigure}[t]{0.45\textwidth}
\begin{center}
\scriptsize
		\tikzstyle{format}=[rectangle,draw,thin,fill=white]
		\tikzstyle{test}=[diamond,aspect=2,draw,thin]
		\tikzstyle{point}=[coordinate,on grid,]
\begin{tikzpicture}[>=stealth]
 \draw[->](-2,0)--(2,0)node[below]{$x$};
 \draw[->](0,-1.5)--(0,1.5)node[right]{$\nabla g_{\lambda}$};
 \node at(-0.2,-0.2){$O$};
 \draw [domain=-0.5:0.5,red] plot(\x,\x)node[above]{$\nabla g_{\lambda}=x$};
 \draw [domain=-1:-0.5,blue] plot(\x,-3*\x-2)node[left]{$\nabla g_{\lambda}=-3x-2$};
  \draw [domain=0.5:1,blue] plot(\x,-3*\x+2)node[right]{$\nabla g_{\lambda}=-3x+2$};
\end{tikzpicture}
\end{center}
 \end{subfigure}
\caption{$g_{\lambda}(x)$ and $\nabla g_{\lambda}(x)$ of Example~\ref{exp:nvip_1}}\label{fig:01}
\end{figure}
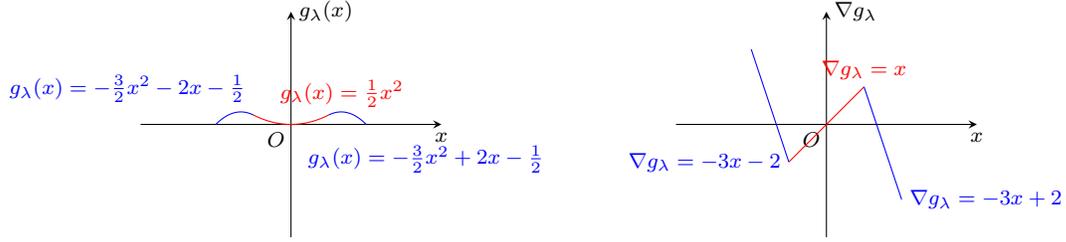
}
\end{example}

\begin{example}\label{exp:nvip_2} 
{\rm 
$F(x)=(2x_1x_2,-x_1^2)^{\top}$ and $
\setX =\RR\times(-\infty,0]$. It is easy to check that the solution set is $
\setX_{\text{\tiny VIP}}^* 
=\{(0,x_2)^{\top} : x_2 \in(-\infty,0]\}\cup\{(x_1,0)^{\top} : x_1 \in(-\infty,0)\}$.
First, 
for $x
=(0,0)^{\top}$ and $x'=
(1,-1)^{\top}$, we have $\langle F(x)-F(x'), x-x'\rangle=\langle (0,0)-(-2,-1), (0,0)-(1,-1)\rangle=-1<0$. Thus, $F(z)$ is non-monotone. 
%
Next, we 
show that the solution set of weak MVI is empty. Let us show this by contradiction. Suppose that there is $\rho>0$ such that $\langle F(x), (x_1,x_2)^{\top}-(0,x_2^*)^{\top}\rangle=x_1^2(x_2+x_2^*)\geq-\frac{\rho}{2}\|F(x)\|^2=-\frac{\rho}{2}x_1^2(4x_2^2+x_1^2)$, $\forall x=(x_1,x_2)^{\top}\in\RR\times(-\infty,0]$ and $\forall x^*=(0,x_2^*)\in\mathbb{X}^*$. If $x_1\neq0$, then $x_1^2(4x_2^2+x_1^2)>0$ and $\delta=\frac{x_2+x_2^*}{4x_2^2+x_1^2}\geq-\frac{\rho}{2}$. However, if $(x_1,x_2)\in(0,\infty)\times(-\infty,0)$ and $x_1\rightarrow0$ and $x_2\rightarrow0$, then $\delta\rightarrow-\infty$, yielding a contradiction with $\Delta\geq-\frac{\rho}{2}$.

Furthermore, suppose that there is $\rho'>0$ such that $\langle F(x), (x_1,x_2)^{\top}-(x_1^*,0)^{\top}\rangle=x_1x_2(x_1-2x_1^*)\geq-\frac{\rho'}{2}\|F(x)\|^2=-\frac{\rho'}{2}x_1^2(4x_2^2+x_1^2)$, $\forall x=(x_1,x_2)^{\top}\in\RR\times(-\infty,0]$ and $\forall x^*=(x_1^*,0)\in\mathbb{X}^*$. If $x_1\neq0$, then $x_1^2(4x_2^2+x_1^2)>0$ and $\delta=\frac{x_2(x_1-2x_1^*)}{x_1(4x_2^2+x_1^2)}\geq-\frac{\rho'}{2}$. However, if $(x_1,x_2)\in(0,\infty)\times(-\infty,0)$ and $x_1\rightarrow0$, then $\delta\rightarrow-\infty$, yielding a contradiction with $\Delta\geq-\frac{\rho'}{2}$. 
Therefore, the solution set of weak MVI is empty.
}
\end{example}

Motivated by these examples, in this paper we shall discuss the proximal gradient method based on the gap function for non-monotone VIP without requiring the MVI or weak MVI conditions.

\subsection{Existing results on the gap functions}

\paragraph{Smooth gap functions and proximal gradient method based on gap function.}
A number of gap functions have been proposed in the literature; see e.g.~\cite{F92, KF98, LP07}. In~\cite{ZM98}, the authors used a smooth gap function to obtain a descent direction and designed the descent method for strongly monotone VIP. This result was extended to monotone VIP in~\cite{ZM93}. In~\cite{LP07}, the authors introduced a new penalty function for the problem of minimizing a twice continuously differentiable function by using the regularized gap function for the variational inequalities. Under certain assumptions, they showed that any stationary point of the new penalty function satisfies the first-order optimality condition of the original constrained minimization problem, and any local (or global) minimizer of the new penalty function is a locally (or globally) optimal solution of the original optimization problem.

\paragraph{Uniform error bounds and level-set error bounds.} The authors of~\cite{ZDLZ21} introduced the local level-set subdifferential error bound and the level-set subdifferential error bound on a level-set. They showed that the local level-set subdifferential error bound is weaker than the Kurdyka-{\/L}ojasiewicz
property and can replace the latter to establish linear convergence for various first-order algorithms. Additionally, they also established the relationships among various local error bounds. Moreover, they showed that the level-set subdifferential error bound on a level-set can lead to a linear convergence to the optimal solution of an additive composite optimization. Recently,~\cite{WLLL23} studied the behavior of the uniform level-set subdifferential error bound via Moreau envelopes
under suitable assumptions. The authors established the relationship among the uniform level-set subdifferential error bound, the uniform Kurdyka-{\/L}ojasiewicz
property, and the uniform H\"older error bound.
Similarly, \cite{KNS16} established the relationships among a series of global error bounds including the Polyak-{\/L}ojasiewicz inequality (may also be viewed as the global Kurdyka-{\/L}ojasiewicz property), global quadratic growth, and so on. Recently,~\cite{LDZ23} establish the relationships among error bounds on a level-set.

\subsection{Main contributions and outline of the paper}

In summary, the main contributions of this paper are as follows:

{\rm(i)} {\bf Theory of Error Bounds.} We establish a fundamental connection between uniform error bounds and level-set error bounds for composite optimization, proving that the former implies the latter - revealing a new hierarchical relationship in the error bound theory that were previously unrecognized.

{\rm(ii)} {\bf Convergence Conditions for PG on Gap Function.} We establish linear convergence of proximal gradient algorithms for gap function minimization in variational inequalities under three key requirements: (i) level-set error bounds, (ii) gradient Lipschitz continuity, and (iii) appropriate initial solution selection. Furthermore, we characterize some fundamental VI 
properties that guarantee conditions (i) and (ii) to hold, thus providing verifiable sufficient conditions on VI for convergence. 

{\rm(iii)} {\bf Initialization-Free Homotopy Method.} For affine variational inequalities, we develop a homotopy continuation method that relaxes the requirement for carefully selected initial points while maintaining global convergence - our method automatically identifies appropriate initialization through path following, overcoming a major computational bottleneck.




The remainder of this paper is organized in the following way. In Section~\ref{sec:PG-OP}, we provide some necessary preliminaries for our analysis. Specifically, Subsection~\ref{sec:pre} is devoted to the weak convexity and subdifferential for a function. In Subsection~\ref{subsec:composite_op}, we introduce the distance-like mappings and functions for composite optimization, and derive their properties. In Subsections~\ref{sec:converge_pg}, we analyze the convergence and linear convergence to the solution of the Proximal Gradient method (PG) for composite optimization. In Subsection~\ref{sec:eb}, we study the behavior of the uniform error bounds and the level-set error bounds. We establish the relationship between the uniform error bounds and the level-set error bounds. Next, in Section~\ref{sec:gap_prop},  we derive the linear convergence to zero of the gap function of VIP. Section~\ref{sec:linear-VIP} identifies a new class of VIP that can be solved via the Proximal Gradient method on its gap function. Finally, in Section~\ref{sec:NE} we provide some  numerical results and experiments on the algorithms proposed in this paper.

\section{Composite Optimization, Proximal Gradient Methods, and Error Bound Conditions}\label{sec:PG-OP}
In order to eventually solve (OP-$g_{\lambda}$), in this section we take a D-tour to conduct a thorough study on the proximal gradient approach to solve non-convex composite optimization models. 
We shall investigate under what conditions  the proximal gradient algorithm will converge linearly -- at least locally -- to an optimal solution of a general non-convex composite optimization model. 
With this in mind, let us start with some preparations. 

\subsection{Weak convexity}\label{sec:pre}

A function $\varphi:\RR^d\rightarrow\RR\cup\{\infty\}$ is $\rho$-weakly convex if
$\psi = \varphi  + \frac{\rho}{2}\|\cdot\|^2$ is convex (\cite{Vial83}). The class of weakly convex functions includes convex function, gradient Lipschitz functions, and many other classes of functions in modern machine learning applications.

For this function class, gradient may not always exist. We define the Fr\'echet subdifferential which is standard in variational analysis~\cite{Rockafellar}.
\begin{definition}
Let $\psi:\RR^d\rightarrow\RR\cup\{\infty\}$ be a proper l.s.c.\ function. 
\begin{itemize}
\item [{\rm(1)}] {\rm(}{\bf Fr\'echet subdifferential}{\rm)} For each $\bar{x}\in\dom\psi$, the Fr\'echet subdifferential of $\psi$ at $\bar{x}$, denoted as $\partial_F\psi(\bar{x})$, is the set of vectors $\xi\in\RR^d$, which satisfy
\[
\lim_{x\neq\bar{x}, x\rightarrow\bar{x}}\inf\frac{1}{\|x-\bar{x}\|}[\psi(x)-\psi(\bar{x})-\langle\xi,x-\bar{x}\rangle\geq0].
\]
If $x\notin\dom\psi$, then $\partial_F\psi(x)=\emptyset$.
\item [{\rm(2)}] {\rm(}{\bf Limiting subdifferential}{\rm)} The limiting-subdifferential, or simply the subdifferential for short, of $\psi$ at $\bar{x}\in\dom\psi$, denoted as $\partial_L{\psi}(\bar{x})$, is defined as follows:
\[
\partial_L\psi(\bar{x}):=\{\xi\in\RR^d:\exists x_d\rightarrow\bar{x}, \psi(x_d)\rightarrow\psi(\bar{x}), \xi_d\in\partial_F\psi(x_d)\rightarrow\xi\}.
\]
\item [{\rm(3)}] {\rm(}{\bf Proximal subdifferential}{\rm)} The proximal subdifferential of $\psi$ at $\bar{x}\in\dom\psi$ written $\partial_P\psi(\bar{x})$, is defined as follows:
\[
\partial_P\psi(\bar{x}):=\{\xi\in\RR^d:\exists\rho, \eta>0, \mbox{ such that } \psi(x)\geq\psi(\bar{x})+\langle\xi,x-\bar{x}\rangle-\rho\|x-\bar{x}\|^2, \forall x\in\setB(\bar{x},\eta) \},
\]
where $\setB(\bar{x},\eta)$ is the open ball of radius $\eta>0$, centered at $\bar{x}$.
\end{itemize}
\end{definition}
Note that for a proper l.s.c.\ weakly convex function $\psi$, the Fr\'echet subdifferential $\partial_F(\bar{x})$ at $\bar{x}$ coincides with the limiting subdifferential $\partial_L(\bar{x})$ and proximal subdifferential $\partial_P(\bar{x})$ at $\bar{x}$ (cf.~\cite{ZDLZ21}). Moreover, the weakly convex function is proximal regular.
Throughout this paper, we simply denote $\partial\psi(\bar{x})=\partial_F\psi(\bar{x})$.
It is also well known that the $\rho$-weak convexity of $\varphi$ is equivalent to (cf.~\cite[Proposition 4.8]{Vial83}): 
\begin{equation}\label{eq:wcvx inequality}
\varphi(y)\geq\varphi(x)+\langle\nu,y-x\rangle-\frac{\rho}{2}\|x-y\|^2
\end{equation}
 for all $x,y\in \RR^d$ and $\nu \in \partial \varphi (x)$.
Furthermore, we note that the $L$-Lipschitz gradient property of $f$ implies that $f(x)\geq f(y)+\langle s,x-y\rangle-\frac{L}{2}\|x-y\|^2$, $\forall x,y\in\RR^d$ with $s\in\partial f(x)$ (see~\cite{nesterov03}, Lemma 1.2.3). This, together with~\cite{Vial83}, establishes that $f$ is $L$-weakly convex.

\subsection{Composite optimization}\label{subsec:composite_op}
We consider the following general non-convex composite optimization problem:
\begin{equation}\label{eq:opt}
\mbox{\rm(OP)}\quad\min_{x\in\RR^d}\phi(x)=f(x)+r(x),
\end{equation}
where $f:\RR^d\rightarrow(-\infty,\infty]$ is a proper lower semi-continuous (l.s.c.) function that is smooth in $\dom f$, and $r:\RR^d\rightarrow(-\infty,\infty]$ is a proper l.s.c.\ function. Let $\setX_{\text{\tiny OP}}^*$ and $\overline{\setX}_{\text{\tiny OP}}$ be the solution set and the critical point set of (OP) respectively. Furthermore, assume that $\setX_{\text{\tiny OP}}^*\neq\emptyset$ and $\phi(x)$ is bounded below. Let $\phi^*=\phi(x^*)$, $\forall x^*\in\setX_{\text{\tiny OP}}^*$. The stationary point of (OP) is assumed to satisfy the following first-order optimality condition: $0\in\partial\phi(x)=\nabla g_{\lambda}(x)+\partial r(x)$. 
If (OP) satisfies the following standard assumption, then (OP) is known as {\it weakly convex composite optimization}.
\begin{assumption}\label{assump1}
$\nabla f$ is $L$-Lipschitz and $r$ is $\rho$-weakly convex.
\end{assumption}

Below, let us introduce some distance-like mappings and functions, which will be crucial in our linear convergence analysis for 
proximal gradient 
in Subsection~\ref{sec:converge_pg}.

{\bf Proximal envelope function:} Let $E_{\alpha}$ be defined as
\begin{equation}
E_{\alpha}(x):=\min_{y\in\RR^d}
\left[ f(x)+\langle\nabla f(x),y-x\rangle+r(y)+\frac{1}{2\alpha}\|x-y\|^2\right] \;\forall x\in\RR^d,
\end{equation}
which is the value function of optimization problem in 
Algorithm~\ref{alg:PGOP}, replacing $x^k$ by $x$.

{\bf Bregman proximal mapping:} Let $T_{\alpha}$ be defined as
\begin{equation}\label{defi:Tk}
T_{\alpha}(x):=\arg\min_{y\in\RR^d} \left[ \langle\nabla f(x),y-x\rangle+r(y)+\frac{1}{2\alpha}\|x-y\|^2\right] \;\forall x\in\RR^d,
\end{equation}
which is the set of optimizers of optimization problem in 
Algorithm~\ref{alg:PGOP}. By Assumption~\ref{assump1} and $\alpha\in(0,\frac{1}{\rho})$, $T_{\alpha}(x)$ is non-empty and a point-to-point mapping. 

{\bf Bregman proximal gap function:} Let $G_{\alpha}$ be defined as
\begin{equation}
G_{\alpha}(x):=-\frac{1}{\alpha}\min_{y\in\RR^d}
\left[ \langle\nabla f(x),y-x\rangle+r(y)-r(x)+\frac{1}{2\alpha}\|x-y\|^2\right] \;\forall x\in\RR^d.
\end{equation}

Obviously, we have $G_{\alpha}(x)\geq0$ for all $x$. Since $r$ is weakly-convex, the following optimization problem is equivalent to the differential inclusion problem $0\in\partial\phi(x)$ associated with problem (OP) (see Proposition~\ref{prop}):
\[
\min_{x\in\RR^d}G_{\alpha}(x).
\]

Since the above mappings/functions are well structured, they satisfy some 
useful properties that are summarized in Proposition~\ref{prop} and Lemma~\ref{lemma:1}, 
whose proofs will be 
relegated to Appendices~\ref{app:prop} and~\ref{app:lem}.

\begin{proposition}\label{prop} {\bf (Properties of mappings and functions)}
Suppose that Assumption~\ref{assump1} holds. Choose $\alpha\in(0,\min\{1/L,1/\rho\})$. Then, for any $x\in\RR^d$ we have: 
\begin{itemize}
\item[{\rm(i)}] $T_{\alpha}(x)$ is a point-to-point mapping;
\item[{\rm (ii)}] $E_{\alpha}(x)=\phi(x)-\alpha G_{\alpha}(x)$;
\item[{\rm (iii)}] $\phi\left(T_{\alpha}(x)\right)\leq\phi(x)-\frac{1}{2}\left(\frac{2}{\alpha}-L-\rho\right)\|x-T_{\alpha}(x)\|^2$;
\item[{\rm (iv)}] $\frac{1-\alpha\rho}{2\alpha^2}\|x-T_{\alpha}(x)\|^2\leq G_{\alpha}(x)$;
\item[{\rm(v)}] $G_{\alpha}(x)\leq\frac{1}{2(1-\alpha\rho)}\dist^2\left(0,\partial\phi(x)\right)$;
\item[{\rm(vi)}] $\|x-T_{\alpha}(x)\|\leq\left(\frac{\alpha}{1-\alpha\rho}\right)\dist\left(0,\partial\phi(x)\right)$;
\item[{\rm(vii)}] $\dist\left(0,\partial\phi(T_{\alpha}(x))\right)\leq\left(L+\frac{1}{\alpha}\right)\|x-T_{\alpha}(x)\|$.
\end{itemize}
\end{proposition}

Under assumptions of Proposition~\ref{prop}, if $\alpha<\min\{1/L,1/\rho\}$, it is easy to show that functions $E_{\alpha}(x)$ and $G_{\alpha}(x)$ are continuous, and the mapping $T_{\alpha}(x)$ is closed and continuous.

\begin{lemma}[Generalized descent inequality]\label{lemma:1}
Suppose Assumption~\ref{assump1} holds. Choose $\alpha\in(0,\min\{1/L,1/\rho\})$. Then, for any $x\in\RR^d$ and  $u\in\RR^d$ we have 
\begin{equation}\label{eq:descent}
\phi(T_{\alpha}(x))-\phi(u)\leq\frac{1}{2}\left[\left(\frac{1}{\alpha}+L\right)\|x-u\|^2-\left(\frac{1}{\alpha}-L\right)\|x-T_{\alpha}(x)\|^2-\left(\frac{1}{\alpha}-\rho\right)\|u-T_{\alpha}(x)\|^2\right].
\end{equation}
\end{lemma}


\subsection{Proximal Gradient (PG) method for solving (OP)}\label{sec:converge_pg}
Let 
$\alpha>0$. The {\it proximal gradient}\/ (PG) method for (OP) 
is formalized as Algorithm~\ref{alg:PGOP} (cf.~\cite{ZDLZ21}). 
\begin{algorithm}[h]
\caption{Proximal Gradient Method for Solving (OP)}
{\bf Initialization:}  $x^0\in\setX$;
\begin{algorithmic}[1]
\For{$k=0,1,\ldots$}
\State Update  $x^{k+1}:=\arg\min_{x\in\RR^d}\left[\langle\nabla f(x^k),x-x^k\rangle+r(x)+\frac{1}{2\alpha}\|x-x^k\|^2\right]$.
\EndFor
\end{algorithmic}
\label{alg:PGOP}
\end{algorithm}

Following~\cite{ZDLZ21}, one 
obtains convergence of 
Algorithm~\ref{alg:PGOP} 
for problem (OP) under  Assumption~\ref{assump1} alone. Specifically, we have 
\[
\lim\limits_{k\rightarrow\infty}\dist(0,\partial\phi(x^{k+1}))=0\quad\mbox{and}\quad\min\limits_{1\leq j\leq k}\dist(0,\partial\phi(x^{j+1}))=o(\sqrt{1/k}).
\]
What happens next is that we shall sharpen the convergence results under additional assumptions. In particular, we first introduce the notion of {\it level-set proximal error bound}\/ ($\ell$-PEB). As we shall see in Theorem~\ref{theo:n-s} that   
($\ell$-PEB) is a necessary and sufficient condition for Algorithm~\ref{alg:PGOP} to have local linear convergence to the set $[\phi\leq\phi^*]=\{x\in\RR^d:\phi(x)\leq\phi^*\}$ (or $\setX_{\text{\tiny OP}}^*$) 
for (OP). In the remainder of our discussion, unless otherwise stated we shall denote $\bar{\phi}=\phi(\bar{x})$ for some given $\bar{x}\in\dom\phi$, and denote $[\phi\leq\bar{\phi}]:=\{x\in\RR^d:\phi(x)\leq\phi(\bar{x})\}$, $[\phi<\bar{\phi}]:=\{x\in\RR^d:\phi(x)<\phi(\bar{x})\}$, $[\phi\geq\bar{\phi}]:=\{x\in\RR^d:\phi(x)\geq\phi(\bar{x})\}$, and $[\phi>\bar{\phi}]:=\{x\in\RR^d:\phi(x)>\phi(\bar{x})\}$.

\begin{definition}[Level-set proximal error bound ($\ell$-PEB)]\label{defi:lPEB} The function $\phi$ is said to satisfy the level-set proximal error bound condition on the level set $[\bar{\phi}\leq\phi<\bar{\phi}+\nu]$ with $\bar{\phi}$ and $\nu>0$, if there is $c_1>0$ such that
\[
\dist(x,[\phi\leq\bar{\phi}])\leq c_1\|x-T_{\alpha}(x)\|,\quad\forall x\in[\bar{\phi}\leq\phi\leq\bar{\phi}+\nu],
\]
where $T_{\alpha}(x)$ is defined in Subsection~\ref{subsec:composite_op}.
\end{definition}

\begin{theorem}{\rm(Necessary and sufficient conditions for linear convergence relative to $[\phi\leq\phi^*]$ of 
Algorithm~\ref{alg:PGOP} 
for (OP))}.\label{theo:n-s}
Suppose Assumption~\ref{assump1} holds. Let a sequence $\{x^k\}$ be generated by Algorithm~\ref{alg:PGOP}, and let $\nu>0$ be given. We have: \\
{\rm(i)} For any  initial point $x^0\in[\phi^*<\phi<\phi^*+\nu]$, if ($\ell$-PEB) holds on $[\phi^*\leq\phi<\phi^*+\nu]$ with $c_1>0$, then 
Algorithm~\ref{alg:PGOP} 
with stepsize $\alpha\in\left(0,\frac{1}{(c_1^2+1)(L+\rho)}\right)$ converges linearly with respect to level-set $[\phi\leq\phi^*]$, i.e.,
\begin{equation}\label{x-Q-linear}
\dist\left(x^{k+1},[\phi\leq\phi^*]\right)\leq\beta \dist\left(x^{k},[\phi\leq\phi^*]\right),\quad\forall k\geq0,
\end{equation}
with $\beta\in(0,1)$.\\
{\rm(ii)} If $\alpha\in(0,\min\{1/L,1/\rho\})$ and 
Algorithm~\ref{alg:PGOP} 
converges linearly in the sense of~\eqref{x-Q-linear} with $\beta\in(0,1)$, then $\phi$ satisfies the ($\ell$-PEB) condition on $[\phi^*\leq\phi<\phi^*+\nu]$.
\end{theorem}

The proof of Theorem~\ref{theo:n-s} is relegated to Appendix~\ref{app:n-s}. 

We shall remark that 
Theorem~\ref{theo:n-s} strengthens the related results in~\cite{ZDLZ21}. First of all, Theorem~\ref{theo:n-s} 
removes the level set boundedness condition required in~\cite{ZDLZ21}. Secondly,  
more tangible sufficient conditions for ($\ell$-PEB) are provided in Subsection~\ref{sec:eb}. 

\subsection{A study of uniform error bounds and level-set error bounds}\label{sec:eb}

Since Theorem~\ref{theo:n-s} stipulates that condition ($\ell$-PEB) is key to local linear convergence for Algorithm~\ref{alg:PGOP}, it is natural to study the connections between condition ($\ell$-PEB) and other tangible and well-studied error-bound conditions in the literature. The aim of this subsection is to present a road-map between these conditions, and to identify suitable sufficient conditions under which the smooth gap function for (VIP) would submit itself to the conclusion of Theorem~\ref{theo:n-s}. 
In this subsection, we study the connections of uniform error bounds (cf.~\cite{WLLL23}) and level-set error bounds (cf.~\cite{ZDLZ21, LDZ23}) for the general problem (OP), and show that the uniform error bounds are sufficient for the level-set error bounds (including ($\ell$-PEB)) to hold. 

Let us first introduce the following conditions under the title of 
{\it uniform error bound conditions}. 
\begin{definition}[Uniform error bounds]$\quad$\\
\vspace{-0.5cm}
\begin{itemize}
\item[{\rm(1)}] {\rm(}{\bf Uniform Kurdyka-{\L}ojasiewicz property (u-KL)}{\rm)} Let $\Omega\subset\dom\partial\phi$ be a nonempty set. A proper and closed function: $\phi:\RR^d\rightarrow\RR$ satisfies the uniform Kurdyka-{\L}ojasiewicz property on $\Omega$ with constant $\mu_1>0$ if there is $\delta>0$ and $\nu>0$ such that
 \[
\phi(x)-\phi(\bar{x})\leq\mu_1\dist^2(0,\partial\phi(x))
\]
for all $\bar{x}\in\Omega$ and any $x$ with $\dist(x,\Omega)<\delta$ and $x\in[\bar{\phi}<\phi<\bar{\phi}+\nu]$.
\item[{\rm(2)}] {\rm(}{\bf Uniform level-set subdifferential error bound (u-SEB)}{\rm)} Let $\Omega\subset\dom\partial\phi$ be a nonempty set. A proper and closed function $\phi:\RR^d\rightarrow\RR$ satisfies the uniform level-set subdifferential error bound on $\Omega$ with constant $\mu_2>0$, if there are $\delta>0$ and $\nu>0$ such that
 \[
\dist(x,[\phi\leq\bar{\phi}])\leq\mu_2\dist(0,\partial\phi(x)),
\]
for all $\bar{x}\in\Omega$ and any $x$ with $\dist(x,\Omega)<\delta$ and $x\in[\bar{\phi}<\phi<\bar{\phi}+\nu]$.
\item[{\rm(3)}] {\rm(}{\bf Uniform H\"older error bound (u-HEB) with $\gamma=2$}{\rm)} Let $\Omega\subset\dom\partial\phi$ be a nonempty set. A proper and closed function: $\phi:\RR^d\rightarrow\RR$ satisfies the uniform H\"older error bound (u-HEB) with $\gamma=2$ on $\Omega$ with constant $\mu_3>0$, if there are $\delta>0$ and $\nu>0$ such that
 \[
\dist^2(x,[\phi\leq\bar{\phi}])\leq\mu_3[\phi(x)-\phi(\bar{x})],
\]
for all $\bar{x}\in\Omega$ and any $x$ with $\dist(x,\Omega)<\delta$ and $x\in[\bar{\phi}<\phi<\bar{\phi}+\nu]$.
\end{itemize}
\end{definition}
Proposition 2.2 of~\cite{WLLL23} states that if $\phi:\RR^d\rightarrow\RR\cup\{\infty\}$ is a proper and closed function and is weakly convex, $\Omega$ is compact and $\phi$ takes a constant value on $\Omega$, then the above-mentioned uniform error bounds on $\Omega$ are equivalent; that is,
\[
\begin{array}{ccccc}
\mbox{(u-KL)}&\Longleftrightarrow&\mbox{(u-SEB)}&\Longleftrightarrow&\mbox{(u-HEB)}.
\end{array}
\]

To complete the picture regarding the local linear convergence properties of the proximal gradient method, next we introduce a different set of conditions under the title of {\it level-set related error bound conditions}. As we shall see later, they are intricately related to the uniform error bound conditions that we introduced earlier. 


\begin{definition}[Level-set error bounds]\label{defi:bebs}$\quad$\\
\vspace{-0.8cm}
\begin{itemize}
\item[{\rm(1)}] {\rm(}{\bf Level-set subdifferential error bound ($\ell$-SEB)}{\rm)} The function $\phi$ is said to satisfy the level-set subdifferential error bound condition on the level set $[\bar{\phi}\leq\phi<\bar{\phi}+\nu]$ with $\bar{\phi}$ and $\nu>0$, if there is $c_2>0$ such that
\[
\dist(x,[\phi\leq\bar{\phi}])\leq c_2\dist(0,\partial\phi(x)),\quad\forall x\in[\bar{\phi}\leq\phi<\bar{\phi}+\nu].
\]
\item[{\rm(2)}] {\rm(}{\bf Level-set Kurdyka-{\L}ojasiewicz property ($\ell$-KL)}{\rm)} The function $\phi$ is said to satisfy the KL property on the level set $[\bar{\phi}\leq\phi<\bar{\phi}+\nu]$ with $\bar{\phi}$ and $\nu>0$, if there is $c_3>0$ such that
\[
c_3[\phi(x)-\bar{\phi}]\leq\dist^2(0,\partial\phi(x)),\quad\forall x\in[\bar{\phi}\leq\phi<\bar{\phi}+\nu].
\]
\item[{\rm(3)}] {\rm(}{\bf Level-set proximal gap-Polyak-{\L}ojasiewic inequality ($\ell$-PGAP)}{\rm)} The function $\phi$ is said to satisfy the proximal gap-PL inequality on the level set $[\bar{\phi}\leq\phi<\bar{\phi}+\nu]$ with $\bar{\phi}$ and $\nu>0$, if there is $c_4>0$ such that
\[
G_{\alpha}(x)\geq c_4[\phi(x)-\bar{\phi}],\quad\forall x\in[\bar{\phi}\leq\phi<\bar{\phi}+\nu],
\]
where $G_{\alpha}(x)$ is defined in Subsection~\ref{subsec:composite_op}.
\item[{\rm(4)}] {\rm(}{\bf Level-set H\"older error bound with $\gamma=2$ ($\ell$-HEB)}{\rm)} The function $\phi$ is said to satisfy the H\"older error bound condition on the level set $[\bar{\phi}\leq\phi<\bar{\phi}+\nu]$ with $\bar{\phi}$ and $\nu>0$, if there is $c_5>0$ such that
\[
c_5\dist^2(x,[\phi\leq\bar{\phi}])\leq\phi(x)-\bar{\phi},\quad\forall x\in[\bar{\phi}\leq\phi<\bar{\phi}+\nu].
\]
\end{itemize}
\end{definition}

The following proposition establishes the relationships among the above level-set error bounds and ($\ell$-PEB). The proof of this proposition is relegated to  Appendix~\ref{app:rela_leb}
\begin{proposition}\label{prop:beb}
Let $\nu>0$, and suppose that the level-set error bounds 
in Definitions~\ref{defi:lPEB} and~\ref{defi:bebs} hold on the set $[\bar{\phi}\leq\phi<\bar{\phi}+\nu]$. Then, the following statements hold true:

{\rm(i)} {\rm($\ell$-PEB)} $\Longleftrightarrow$ {\rm($\ell$-SEB)}.

{\rm(ii)} {\rm($\ell$-PEB)} $\Longrightarrow$ {\rm($\ell$-PGAP)} $\Longrightarrow$ {\rm($\ell$-KL)}.
\end{proposition}
The relationships of the level-set error bounds discussed in Proposition~\ref{prop:beb} is visualized in Figure~\ref{fig:1}.
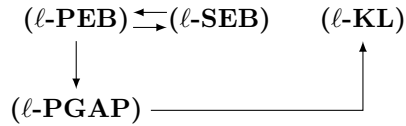
\begin{figure}[h]
\centering
\begin{center}
\scriptsize
		\tikzstyle{format}=[rectangle,draw,thin,fill=white]
		\tikzstyle{test}=[diamond,aspect=2,draw,thin]
		\tikzstyle{point}=[coordinate,on grid,]
\begin{tikzpicture}
[
>=latex,
node distance=5mm,
 ract/.style={draw=blue!50, fill=blue!5,rectangle,minimum size=6mm, very thick, font=\itshape, align=center},
 racc/.style={rectangle, align=center},
 ractm/.style={draw=red!100, fill=red!5,rectangle,minimum size=6mm, very thick, font=\itshape, align=center},
 cirl/.style={draw, fill=yellow!20,circle,   minimum size=6mm, very thick, font=\itshape, align=center},
 raco/.style={draw=green!500, fill=green!5,rectangle,rounded corners=2mm,  minimum size=6mm, very thick, font=\itshape, align=center},
 hv path/.style={to path={-| (\tikztotarget)}},
 vh path/.style={to path={|- (\tikztotarget)}},
 skip loop/.style={to path={-- ++(0,#1) -| (\tikztotarget)}},
 vskip loop/.style={to path={-- ++(#1,0) |- (\tikztotarget)}}]

        \node (a) [racc]{\baselineskip=3pt\small {\bf ($\ell$-SEB)}};
        \node (b) [racc, left of = a, xshift=-40]{\baselineskip=3pt\small {\bf ($\ell$-PEB)}};
        \node (c) [racc, right of = a, xshift=40]{\baselineskip=3pt\small {\bf ($\ell$-KL)}};
        \node (d) [racc, below of = b, yshift=-20]{\baselineskip=3pt\small {\bf ($\ell$-PGAP)}};
        \path 
              (-0.7,-0.1) edge[<-] (-1.15,-0.1)
              (-0.7,0.1) edge[->] (-1.15,0.1)
              (b) edge[->] (d);
        \path (d) edge[->, hv path] (c);
\end{tikzpicture}
\caption{The relationships among the notions of the level-set error bounds}\label{fig:1}
\end{center}
\end{figure}
In the rest of this subsection, we consider the specific level-set error bounds on $[\phi^*\leq\phi<\phi^*+\nu]$. We call the {\rm($\ell$-HEB)} on $[\phi^*\leq\phi<\phi^*+\nu]$ as the level-set quadratic growth condition {\rm($\ell$-QG)}. Next definition will introduce the definition of {\rm($\ell$-QG)} and restricted secant inequality (RSI).
\begin{definition}\label{defi:qg}
\begin{itemize}
\item[{\rm(1)}] {\rm({\bf Level-set quadratic growth ($\ell$-QG)})} The function $\phi$ is said to satisfy the level-set quadratic growth with $\phi^*$ and $\nu>0$, if there is $c_5>0$ such that
\[
c_5\dist^2(x,[\phi\leq\phi^*])\leq\phi(x)-\phi^*,\quad\forall x\in[\phi^*\leq\phi<\phi^*+\nu].
\]
\item[{\rm(2)}] {\rm({\bf Restricted secant inequality (RSI)})} The function $\phi$ is said to satisfy the restricted secant inequality with $\phi^*$ and $\nu>0$, if there is $c_6>0$ such that
\[
c_6\dist^2(x,[\phi\leq\phi^*])\leq\langle\xi,x-x_p\rangle,\quad\forall x\in[\phi^*\leq\phi <\phi^*+\nu]\quad\mbox{and}\quad\xi\in\partial\phi(x),
\]
with $x_p=\arg\min_{y\in\setX^*}\|x-y\|$.
\end{itemize}
\end{definition}
Next proposition shows that ($\ell$-PEB) on $[\phi^*\leq\phi<\phi^*+\nu]$ implies ($\ell$-QG) and (RSI) implies the ($\ell$-SEB) on $[\phi^*\leq\phi <\phi^*+\nu]$. The proof of Proposition~\ref{prop:beb-2}{} can be found in Appendix~\ref{app:rela_leb_2}.

\begin{proposition}\label{prop:beb-2}
Let $\nu>0$. Regarding the level-set error bounds in Definition~\ref{defi:bebs} on set $[\phi^*\leq\phi <\phi^*+\nu]$ and Definition~\ref{defi:qg}, the following relationships hold: 

{\rm(i)} {\rm($\ell$-PEB)} on $[\phi^*\leq\phi<\phi^*+\nu]$ $\Longrightarrow$ {\rm($\ell$-QG)}.

{\rm(ii)} {\rm(RSI)} $\Longrightarrow$ {\rm($\ell$-SEB)} on $[\phi^*\leq\phi<\phi^*+\nu]$ (or {\rm($\ell$-PEB)} on $[\phi^*\leq\phi<\phi^*+\nu]$).
\end{proposition}
Recall the optimization problem (OP), and we make the following assumption on $\setX_{\rm\tiny OP}^*$.
\begin{assumption}\label{assump2}
The optimal solution set $\setX_{\rm\tiny OP}^*$ of (OP) is compact.
\end{assumption}

The following two lemmas reveal 
the relationship between the {\it uniform error bound conditions}\/ on $\Omega=\setX_{\text{\tiny OP}}^*$ and the {\it level-set error bound conditions}\/ on $[\phi^*\leq\phi<\phi^*+\nu]$, $x^*\in\setX_{\text{\tiny OP}}^*$, under Assumptions~\ref{assump1} and~\ref{assump2}.
\begin{lemma}\label{lemma:u-l}
Suppose Assumptions~\ref{assump1} and~\ref{assump2} hold for (OP). If one of the three uniform error bounds \{(u-KL), (u-SEB), (u-HEB)\} holds on $\Omega=\setX_{\text{\tiny OP}}^*$ with $\mu>0$, $\nu>0$, and, $\delta>0$, then there is $\nu'>0$ such that the three level-set error bound also hold on $[\phi^*\leq\phi<\phi^*+\nu']$ with the same constant.
\end{lemma}
\begin{proof} Denote $\setA_{\delta}:=\{x\in\RR^d : \: \dist(x,\setX_{\text{\tiny OP}}^*)<\delta\}$, $\setB_{\nu}:=[\phi^*<\phi<\phi^*+\nu]$, and $\bar{\setB}_{\nu}:=[\phi^*\leq\phi<\phi^*+\nu]$ for $\nu>0$ and $\bar{\setB}_{0}:=[\phi^*=\phi]$. 
Since one of the uniform error bounds \{(u-KL), (u-SEB), (u-HEB)\} holds on $\setA_{\delta}\cap\setB_{\nu}$, 
by the fact $\phi(x)-\phi^*=0$ and $\dist(x,[\phi\leq\phi^*])=0$, $\forall x\in[\phi\leq\phi^*]=\setX_{\text{\tiny OP}}^*$, we conclude that 
the error bound condition holds on $\setA_{\delta}\cap\bar{\setB}_{\nu}$.
Setting $\nu_k=\nu/(k+1)$ for $k\in\setN$, we have 
\[
\setA_{\delta}\cap\bar{\setB}_{\nu_{k+1}}\subset\setA_{\delta}\cap\bar{\setB}_{\nu_{k}}\subset\cdots
\subset \setA_{\delta}\cap\bar{\setB}_{\nu_{0}}, \nu_0=\nu,~\mbox{and} \lim_{k\rightarrow\infty}\nu_k=0.
\]
By the fact that $\bar{\setB}_0=\setX_{\text{\tiny OP}}^*$, we have that
\[
\setA_{\delta}\cap\bar{\setB}_{\nu_{k}}\neq\emptyset, \forall k\in\setN, \mbox{ and }\lim_{k\rightarrow\infty}\bar{\setB}_{\nu_k}=\setX_{\text{\tiny OP}}^*.
\]
Suppose that, for every $k\in\setN$, there is $x_k\in\bar{\setB}_{\nu_k}$, and $x_k\notin\setA_{\delta}$ or $\dist(x_k,\setX_{\text{\tiny OP}}^*)\geq\delta$. However, by the fact $\lim\limits_{k\rightarrow\infty}\bar{\setB}_{\nu_k}=\setX_{\text{\tiny OP}}^*$ and that $\phi$ is l.s.c., we have $\lim\limits_{k\rightarrow\infty}\dist(x_k,\setX_{\text{\tiny OP}}^*)=0$, which yields a contradiction with $\dist(x_k,\setX_{\text{\tiny OP}}^*)\geq\delta$, $\forall k\in\setN$.

Consequently, there is $\bar{k}\in\setN$ with $\nu_{\bar{k}}>0$, such that for every $x_{\bar{k}}\in\bar{\setB}_{\nu_{\bar{k}}}$, we have $x_{\bar{k}}\in\setA_{\delta}$, i.e., $\bar{\setB}_{\nu_{\bar{k}}}\subset\setA_{\delta}$ or $\setA_{\delta}\cap\bar{\setB}_{\nu_{\bar{k}}}=\bar{\setB}_{\nu_{\bar{k}}}$. 

Then, if there are $\nu>0$ and $\delta>0$ such that one of the uniform error bounds \{(u-KL), (u-SEB), (u-HEB)\} hold with $\dist(x,\setX_{\text{\tiny OP}}^*)<\delta$ and $x\in[\phi^*<\phi<\phi^*+\nu]$, then the error bound also holds on $[\phi^*\leq\phi<\phi^*+\nu']$ with $\nu'=\nu_{\bar{k}}>0$.
\end{proof}

Conversely, we have the following lemma.

\begin{lemma}\label{lemma:l-u}
Suppose Assumptions~\ref{assump1} and~\ref{assump2} hold for (OP). If one of the three level-set error bound conditions \{($\ell$-KL), ($\ell$-SEB), ($\ell$-QG)\} holds on $[\phi^*\leq\phi<\phi^*+\nu]$ with $c>0$, $\nu>0$, then there is $\delta>0$ and $\nu'>0$ such that the three uniform error bound conditions \{(u-KL), (u-SEB), (u-HEB)\} also hold on $\Omega=\setX_{\text{\tiny OP}}^*$ with these constant. Moreover, the level-set error bound conditions \{($\ell$-KL), ($\ell$-SEB), ($\ell$-QG)\} on $[\phi^*\leq\phi<\phi^*+\nu]$ are equivalent.
\end{lemma}
\begin{proof} For given $\delta>0$, by the proofs for Lemma~\ref{lemma:u-l}, there exists $0<\nu'<\nu$ such that $\setA_{\delta}\cap\bar{\setB}_{\nu'}=\bar{\setB}_{\nu'}$. Thus $\setA_{\delta}\cap\setB_{\nu'}=\setB_{\nu'}=[\phi^*<\phi<\phi^*+\nu']$.

Since one of the level-set error bounds \{($\ell$-KL), ($\ell$-SEB), ($\ell$-QG)\} holds on $[\phi^*\leq\phi<\phi^*+\nu]$ and $0<\nu'<\nu$, then it also holds on $\Omega=\setX_{\text{\tiny OP}}^*$ with $x\in\setA_{\delta}\cap\setB_{\nu'}$. 
Moreover, by Proposition 2.2 of~\cite{WLLL23} and Lemma~\ref{lemma:u-l}, we arrive at the conclusion that the level-set error bounds \{($\ell$-KL), ($\ell$-SEB), ($\ell$-QG)\} are equivalent.
\end{proof}
By combining the results of Proposition~\ref{prop:beb}, Proposition~\ref{prop:beb-2}, Lemma~\ref{lemma:u-l}, Lemma~\ref{lemma:l-u} and Proposition 2.2 in~\cite{WLLL23}, we are in the position to depict the relationships among the level-set error bound conditions on $[\phi^*\leq\phi<\phi^*+\nu]$ in Figure~\ref{fig:2}. 
\begin{figure}[h]
\centering
\begin{center}
\scriptsize
		\tikzstyle{format}=[rectangle,draw,thin,fill=white]
		\tikzstyle{test}=[diamond,aspect=2,draw,thin]
		\tikzstyle{point}=[coordinate,on grid,]
\begin{tikzpicture}
[
>=latex,
node distance=5mm,
 ract/.style={draw=blue!50, fill=blue!5,rectangle,minimum size=6mm, very thick, font=\itshape, align=center},
 racc/.style={rectangle, align=center},
 ractm/.style={draw=red!100, fill=red!5,rectangle,minimum size=6mm, very thick, font=\itshape, align=center},
 cirl/.style={draw, fill=yellow!20,circle,   minimum size=6mm, very thick, font=\itshape, align=center},
 raco/.style={draw=green!500, fill=green!5,rectangle,rounded corners=2mm,  minimum size=6mm, very thick, font=\itshape, align=center},
 hv path/.style={to path={-| (\tikztotarget)}},
 vh path/.style={to path={|- (\tikztotarget)}},
 skip loop/.style={to path={-- ++(0,#1) -| (\tikztotarget)}},
 vskip loop/.style={to path={-- ++(#1,0) |- (\tikztotarget)}}]
       
        \node (a) [racc]{\baselineskip=3pt\small {\bf ($\ell$-SEB)}\\
        {\bf on $[\phi^*\leq\phi<\phi^*+\nu]$}};
         \node (ar) [racc, right of =a, xshift=40, yshift=8]{\baselineskip=3pt\tiny {Lemma~\ref{lemma:l-u}}};
        \node (aa) [racc, above of =a, yshift=40]{\baselineskip=3pt\small {\bf (u-SEB)}\\
        {\bf on $\setX^*$}};
        \node (aat) [racc, above of =aa, xshift=60]{\baselineskip=3pt\tiny { Proposition 2.2 of~\cite{WLLL23}}};
        \node (at) [racc, below of =a, yshift=-35]{\baselineskip=3pt\tiny {Proposition~\ref{prop:beb}}};
        \node (aatt) [racc, below of =aa, xshift=-24, yshift=-15]{\baselineskip=3pt\tiny { Lemma~\ref{lemma:u-l}}};
        \node (aattt) [racc, below of =aa, xshift=20, yshift=-15]{\baselineskip=3pt\tiny { Lemma~\ref{lemma:l-u}}};
        \node (a0) [ractm, above of = aa,yshift=30]{\baselineskip=3pt\small {\bf Assumptions~\ref{assump1} and~\ref{assump2}}};
        \node (b) [racc, left of = a, xshift=-100]{\baselineskip=3pt\small {\bf ($\ell$-PEB)}\\
        {\bf on $[\phi^*\leq\phi<\phi^*+\nu]$}};
        \node (bt) [racc, right of =b, xshift=40, yshift=8]{\baselineskip=3pt\tiny {Proposition~\ref{prop:beb}}};
         \node (btt) [racc, left of =b, xshift=-50, yshift=8]{\baselineskip=3pt\tiny {Proposition~\ref{prop:beb-2}}};
         \node (bttt) [racc, below of =b, xshift=-24, yshift=-12]{\baselineskip=3pt\tiny {Proposition~\ref{prop:beb}}};
        \node (c) [racc, right of = a, xshift=100]{\baselineskip=3pt\small {\bf ($\ell$-KL)}\\
        {\bf on $[\phi^*\leq\phi<\phi^*+\nu]$}};
        \node (cr) [racc, right of =c, xshift=40, yshift=8]{\baselineskip=3pt\tiny {Lemma~\ref{lemma:l-u}}};
        \node (cc) [racc, above of = c, yshift=40]{\baselineskip=3pt\small {\bf (u-KL)}\\
        {\bf on $\setX^*$}};
        \node (cct) [racc, below of =cc, xshift=-24, yshift=-15]{\baselineskip=3pt\tiny { Lemma~\ref{lemma:u-l}}};
        \node (cctt) [racc, below of =cc, xshift=20, yshift=-15]{\baselineskip=3pt\tiny { Lemma~\ref{lemma:l-u}}};
         \node (ccttt) [racc, above of =cc, xshift=55]{\baselineskip=3pt\tiny { Proposition 2.2 of~\cite{WLLL23}}};
        \node (ccc) [racc, right of = cc, xshift=100]{\baselineskip=3pt\small {\bf (u-HEB)}\\
        {\bf on $\setX^*$}};
        \node (ccct) [racc, below of =ccc, xshift=-24, yshift=-15]{\baselineskip=3pt\tiny { Lemma~\ref{lemma:u-l}}};
        \node (ccctt) [racc, below of =ccc, xshift=20, yshift=-15]{\baselineskip=3pt\tiny { Lemma~\ref{lemma:l-u}}};
        \node (d) [racc, below of = b, yshift=-40]{\baselineskip=3pt\small {\bf ($\ell$-PGAP)}\\
        {\bf on $[\phi^*\leq\phi<\phi^*+\nu]$}};
        \node (e) [racc, right of = c, xshift=100]{\baselineskip=3pt\small {\bf ($\ell$-QG)}};
        \node (f) [racc, left of = aa, xshift=-100]{\baselineskip=3pt\small {\bf (RSI)}};
        \node (ft) [racc, below of =f, xshift=-24, yshift=-7]{\baselineskip=3pt\tiny { Proposition~\ref{prop:beb-2}}};
        \path 
              (-1.6,-0.1) edge[<-] (-2.5,-0.1)
              (-1.6,0.1) edge[->] (-2.5,0.1)
              (1.6,2) edge[<-] (2.5,2)
              (1.6,1.8) edge[->] (2.5,1.8)
              (1.6,-0.1) edge[->] (2.5,-0.1)
              (1.6,0.1) edge[<-] (2.5,0.1)
              (5.4,2) edge[<-] (6.3,2)
              (5.4,1.8) edge[->] (6.3,1.8)
              (5.4,0.1) edge[<-] (6.3,0.1)
              (5.4,-0.1) edge[->] (6.3,-0.1)
              (b) edge[->] (d)
              (b) edge[-] (-7,0)
              (-7,-3) edge[-] (-7,0)
              (-0.075,1.5) edge[->] (-0.075,0.4)
              (0.075,1.5) edge[<-] (0.075,0.4)
              (3.925,1.5) edge[->] (3.925,0.4)
              (f) edge[->] (-4,0.7)
              (4.075,1.5) edge[<-] (4.075,0.4);
        \path (d) edge[->, hv path] (c);
        \path (-7,-3) edge[->, hv path] (e);
        \path (7.95,0.35) edge[<-] (7.95,1.5);
        \path (8.1,1.5) edge[<-] (8.1,0.35);
        \draw[dotted,very thick] (-1,1.3) -- (9.25,1.3);
        \draw[dotted,very thick] (-1,2.8) -- (9.25,2.8);
        \draw[dotted,very thick] (-1,1.3) -- (-1,2.8);
        \draw[dotted,very thick] (9.25,1.3) -- (9.25,2.8);
        \draw[dotted,very thick] (-6,0.6) -- (9.25,0.6);
        \draw[dotted,very thick] (-6,-0.6) -- (9.25,-0.6);
        \draw[dotted,very thick] (-6,0.6) -- (-6,-0.6);
        \draw[dotted,very thick] (9.25,0.6) -- (9.25,-0.6);
\end{tikzpicture}
\caption{The relationships among the notions of the level-set error bounds on $[\phi^*\leq\phi <\phi^*+\nu]$}\label{fig:2}
\end{center}
\end{figure}
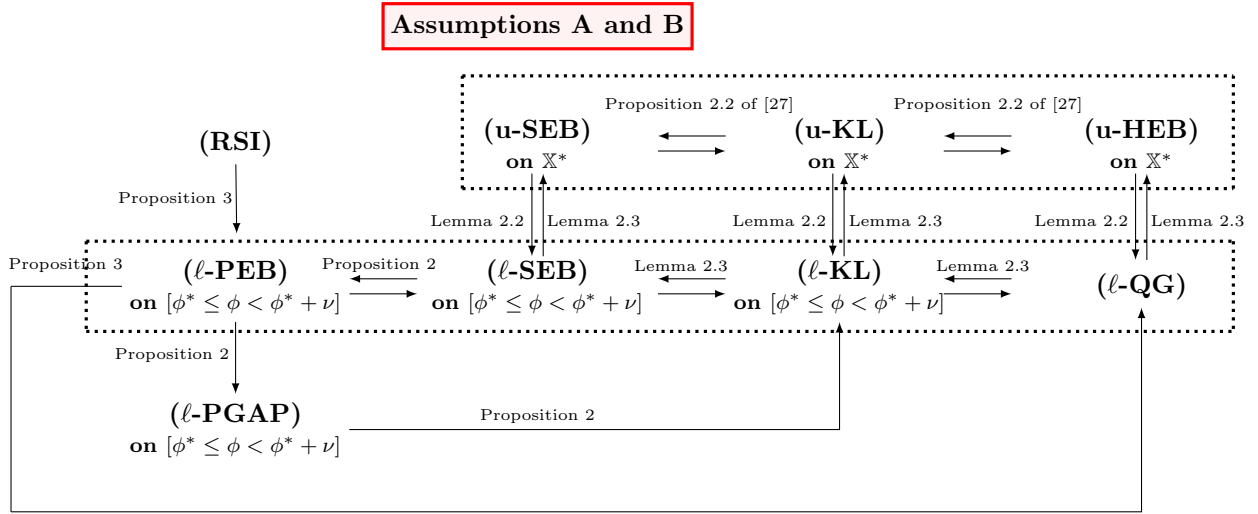

\begin{proposition}\label{prop:eb_set}
Suppose that $\phi$ is level-bounded, i.e., the set $\{x\in\RR^d : \phi(x)\leq r\}$ is bounded (possibly empty) for every $r\in\RR$. Furthermore, suppose that the value of $\min\limits_{x\in\overline{\setX}_{\text{\tiny OP}}\setminus\setX_{\text{\tiny OP}}^*}\phi(x)$ exists and there exists $\sigma>0$ such that $\min\limits_{x\in\overline{\setX}_{\text{\tiny OP}}\setminus\setX_{\text{\tiny OP}}^*}\phi(x)-\phi^*>\sigma$,
where $\overline{\setX}_{\text{\tiny OP}}$ is the set of all critical points and $\setX_{\text{\tiny OP}}^*$ is the set of global minimisers. Suppose that the ($\ell$-PEB) condition holds on $\setX_{\text{\tiny OP}}^*$ with $c_1>0$ and $0<\nu<\sigma$ and the function $\dist(0,\partial\phi)$ is lower semi-continuous at the level set $[\phi^*\leq\phi<\phi^*+\sigma]$. Then, the ($\ell$-PEB) condition also holds on the set $[\phi^*\leq\phi<\phi^*+\sigma]$.
\end{proposition}
\begin{proof} For any $x\in[\phi^*+\nu\leq\phi<\phi^*+\sigma]$, by statement (iii) of Proposition~\ref{prop}, we have $T_{\alpha}(x)\in[\phi^*\leq\phi<\phi^*+\sigma]$.

\textit{Case 1: $\phi(T_{\alpha}(x))=\phi^*$.} Then, $T_{\alpha}(x)\in\setX_{\text{\tiny OP}}^*$, and trivially we have $\dist(x,\setX_{\text{\tiny OP}}^*)\leq\|x-T_{\alpha}(x)\|$; i.e., ($\ell$-PEB) indeed holds on $[\phi^*+\nu\leq\phi<\phi^*+\sigma]$. 

\textit{Case 2: $T_{\alpha}(x)\in[\phi^*<\phi<\phi^*+\sigma]$.} In this case, $\min\limits_{x\in\overline{\setX}_{\text{\tiny OP}}\setminus\setX_{\text{\tiny OP}}^*}\phi(x)-\phi^*>\sigma$ implies that there is no critical point in the set $[\phi^*+\nu<\phi<\phi^*+\sigma]$. By the lower semi-continuous property of $\dist(0,\partial\phi(x))$, we have $\delta>0$ such that $\dist(0,\partial\phi(x))\geq\delta>0$, $\forall x\in[\phi^*+\nu\leq\phi<\phi^*+\sigma]$. Since $\phi$ is level-bounded, the distance $\dist(x,\setX_{\text{\tiny OP}}^*)$ is bounded by $r'$ for any $x\in[\phi^*+\nu\leq\phi<\phi^*+\sigma]$. Thus, we have $\dist(x,\setX_{\text{\tiny OP}}^*)\leq\frac{\delta}{r'}\dist(0,\partial\phi(x))$ on $[\phi^*+\nu\leq\phi<\phi^*+\sigma]$. Additionally, since the ($\ell$-PEB) condition holds on $\setX_{\text{\tiny OP}}^*$ with $c_1>0$ and $0<\nu$, $\forall x\in[\phi^*\leq\phi<\phi^*+\nu]$, we have
\[
\begin{aligned}
\dist(x,\setX_{\text{\tiny OP}}^*)&\leq c_1\|x-T_{\alpha}(x)\|\\
                &\leq c_1\left(\frac{\alpha}{1-\alpha\rho}\right)\dist(0,\partial\phi(x)) \quad\mbox{(by statement (vi) of Proposition~\ref{prop})}
\end{aligned}
\]
Therefore, $\dist(x,\setX_{\text{\tiny OP}}^*)\leq c_2\dist(0,\partial\phi(x))$ with $c_2=\max\left\{\frac{\delta}{r'},c_1\left(\frac{\alpha}{1-\alpha\rho}\right)\right\}$ holds on $[\phi^*<\phi<\phi^*+\sigma]$. Furthermore, $\forall x\in[\phi^*\leq\phi<\phi^*+\sigma]$, we have
\[
\begin{aligned}
\dist(x,\setX_{\text{\tiny OP}}^*)&\leq\dist(T_{\alpha}(x),\setX_{\text{\tiny OP}}^*)+\|x-T_{\alpha}(x)\|\\
&\leq c_2\dist(0,\partial\phi(T_{\alpha}(x)))+\|x-T_{\alpha}(x)\|\\
&\leq\left(c_2(L+\frac{1}{\alpha})+1\right)\|x-T_{\alpha}(x)\|\quad\mbox{(by statement (vii) of Proposition~\ref{prop})},
\end{aligned}
\]
which leads to the conclusion that ($\ell$-PEB) holds on the set $[\phi^*\leq\phi<\phi^*+\sigma]$.
\end{proof}

\section{The Gap Function for (VIP) and Its Properties}\label{sec:gap_prop}

Let us return to the variational inequality problem~\eqref{Prob:VIP} that we set out to solve. Recall the gap function~\eqref{eq:SGF_1} defined for solving~\eqref{Prob:VIP}, and its associated smooth non-convex optimization problem~\eqref{Prob:OP}. In particular, recall that $g_{\lambda}(x)\ge 0$ for all $x\in\setX$, and $g_{\lambda}(x)= 0$ 
$\Longleftrightarrow
x\in \setX_{\text{\tiny VIP}}^*$. Per analysis in Section~\ref{sec:PG-OP}, we know that Assumption~\ref{assump1} for $\phi$ is a basic assumption for the convergence of the proximal gradient (Algorithm~\ref{alg:PGOP}) on (OP). If $g_{\lambda}(x)$ satisfies Assumption~\ref{assump1}, then the sequence $\{\dist(0,\partial\psi(x^k))\}$ generated by 
Algorithm~\ref{alg:PGOP} 
converges to $0$. However, Example~\ref{exp:nvip_1} shows that $x$ satisfying $0\in\partial\psi(x^k))$ does not necessarily guarantee 
$g_{\lambda}(x)=0$. In order to solve (VIP), we need to find a solution $x^*\in\setX_{\text{\tiny VIP}}^*$, or $g_{\lambda}(x^*)=0$ for some $x^*\in\setX$.

Now, Theorem~\ref{theo:n-s} shows that ($\ell$-PEB) is a necessary and sufficient condition for linear convergence of the sequence $\{\dist\left(x^{k},[\psi\leq\psi^*]\right)\}$ generated by Algorithm~\ref{alg:PGOP}, if initiated from a point sufficiently close to the solution set, which is a promising way to solve  
(OP-$g_{\lambda}$). In generally, however, it is still difficult at this point to identify if 
the gap function $g_{\lambda}(x)$ for an VI problem satisfies the level-set error bound ($\ell$-PEB) on $[\psi^*\leq\psi<\psi^*+\nu]$ with some $\nu>0$ or not. The good news is: As we demonstrated in Subsection~\ref{sec:eb} that the combination of Assumption~\ref{assump2} with any one of the three uniform error bounds \{(u-KL), (u-SEB), (u-HEB)\} can guarantee ($\ell$-PEB) to hold.

Therefore, if Assumptions~\ref{assump1} and~\ref{assump2} hold on the problem (OP-$g_{\lambda}$), moreover if one of the three uniform error bounds \{(u-KL), (u-SEB), (u-HEB)\} holds on (OP-$g_{\lambda}$), then Lemma~\ref{lemma:u-l} and Theorem~\ref{theo:n-s} give rise to the following result.

\begin{theorem}[linear convergence of PG based on smooth gap function to solve (VIP)]\label{theo:n-s-VIP}
Suppose that Assumptions~\ref{assump1} and~\ref{assump2} hold for (OP-$g_{\lambda}$). Moreover, suppose one of the three uniform error bounds \{(u-KL), (u-SEB), (u-HEB)\} holds. Let a sequence $\{x^k\}$ be generated by Algorithm~\ref{alg:PGOP}. Then, there exists positive number $\nu>0$ such that for any initial point $x^0\in[\psi^*<\psi<\psi^*+\nu]$, Algorithm~\ref{alg:PGOP} converges linearly with respect to level-set $[\psi\leq\psi^*]$; i.e., there is $\beta\in(0,1)$ such that
    \[
    \dist(x^{k+1},[\psi\leq\psi^*])\leq\beta\dist(x^k,[\psi\leq\psi^*]),\quad\forall k\geq0.
    \]
\end{theorem}
\begin{proof}
The result follows by combining 
Lemma~\ref{lemma:u-l}, Proposition~\ref{prop:beb}, 
and Theorem~\ref{theo:n-s}.
\end{proof}
By the results of Theorem~\ref{theo:n-s-VIP}, we conclude that the following three are the basic assumptions to guarantee local linear convergence of Algorithm~\ref{alg:PGOP} 
to a solution of (VIP): 
\begin{itemize}
    \item[{\rm 1.}] $g_{\lambda}(x)$ satisfies Assumption~\ref{assump1};
    \item[{\rm 2.}] The solution set of (VIP) satisfy Assumption~\ref{assump2};
    \item[{\rm 3.}] One of the three uniform error bounds \{(u-KL), (u-SEB), (u-HEB)\} holds.    
\end{itemize}
In the next section, we will expound on which type of VI problems will satisfy the three aforementioned conditions.

\section{Identifying New Classes of (VIP) That Can Be Solved via Proximal Gradient Methods on Its Gap Functions}\label{sec:linear-VIP}

In this section, we will propose some sufficient conditions of the gradient Lipschitz of $g_{\lambda}(x)$ in Subsection~\ref{app:suff-Lips-gap}. The sufficient conditions for the compactness of the solution set of (VIP) will be proposed in Subsection~\ref{app:suff-compact}. The sufficient conditions for one of the three error bounds \{(u-KL), (u-SEB), (u-HEB)\} will be illustrated in Subsection~\ref{subsec:suff_ebs}.
The results from the properties of (VIP) to the linear convergence of PG for solving (VIP) are concluded in Figure~\ref{fig:3}. 
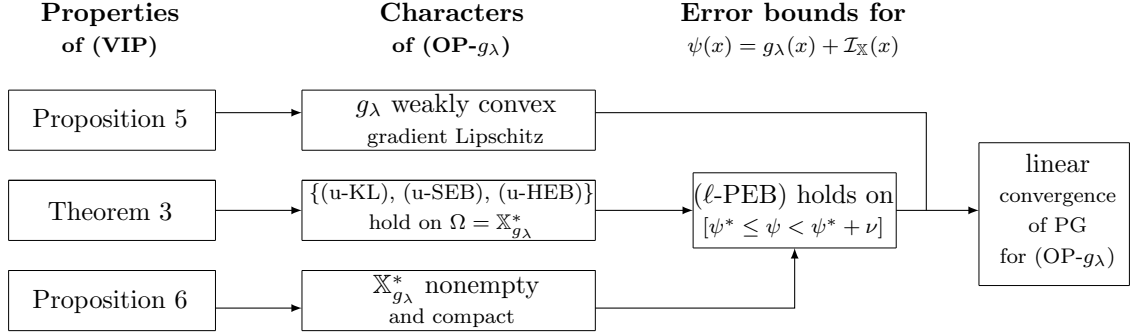
\begin{figure}
\centering
\begin{center}
\scriptsize
		\tikzstyle{format}=[rectangle,draw,thin,fill=white]
		\tikzstyle{test}=[diamond,aspect=2,draw,thin]
		\tikzstyle{point}=[coordinate,on grid,]
\begin{tikzpicture}
[
>=latex,
node distance=5mm,
 ract/.style={draw=blue!50, fill=blue!5,rectangle,minimum size=6mm, very thick, font=\itshape, align=center},
 racc/.style={rectangle, align=center},
 ractm/.style={draw=red!100, fill=red!5,rectangle,minimum size=6mm, very thick, font=\itshape, align=center},
 cirl/.style={draw, fill=yellow!20,circle,   minimum size=6mm, very thick, font=\itshape, align=center},
 raco/.style={draw=green!500, fill=green!5,rectangle,rounded corners=2mm,  minimum size=6mm, very thick, font=\itshape, align=center},
 hv path/.style={to path={-| (\tikztotarget)}},
 vh path/.style={to path={|- (\tikztotarget)}},
 skip loop/.style={to path={-- ++(0,#1) -| (\tikztotarget)}},
 vskip loop/.style={to path={-- ++(#1,0) |- (\tikztotarget)}}]

        \node (a) [racc]{\baselineskip=3pt\small {\bf Properties}\\{\bf of (VIP)}};
        \node (c) [racc, right of = a, xshift=115]{\baselineskip=3pt\small {\bf Characters}\\{\bf of (OP-$g_{\lambda}$)}};
        \node (c1) [racc, right of = c, xshift=115]{\baselineskip=3pt\small {\bf Error bounds for}\\{\bf $\psi(x)=g_{\lambda}(x)+\mathcal{I}_{\setX}(x)$}};
        \node (aa) [racc, below of = a, yshift=-20]
        {\baselineskip=3pt\small {Proposition~\ref{prop:suff-grad-Lips}}};
        \node (cc) [racc, below of = c, yshift=-20]
        {\baselineskip=3pt\small { $g_{\lambda}$ weakly convex}
        \\ { gradient Lipschitz}};
        \node (aaa) [racc, below of = aa, yshift=-20]
        {\baselineskip=3pt\small {Theorem~\ref{prop:suff-HEB}}};
        \node (ccc) [racc, below of = cc, yshift=-20]
        {\baselineskip=3pt\small {\scriptsize\{(u-KL), (u-SEB), (u-HEB)\}} \\
        {hold on} {$\Omega=\setX_{g_{\lambda}}^*$}};
        \node (aaaa) [racc, below of = aaa, yshift=-20]
        {\baselineskip=3pt\small {Proposition~\ref{prop:suff-compact-1}}};
        \node (cccc) [racc, below of = ccc, yshift=-20]
        {\baselineskip=3pt\small { $\setX_{g_{\lambda}}^*$ nonempty}\\
        {and compact}};
        \node (d2) [racc, right of = ccc, xshift=115]
        {\baselineskip=3pt\small {($\ell$-PEB)} {holds on}\\
        {$[\psi^*\leq\psi<\psi^*+\nu]$}};
        \node (e) [racc, right of = ccc, xshift=215]
        {\baselineskip=3pt\small {linear}\\ {convergence}\\
        {of PG}\\ {for (OP-$g_{\lambda}$)}};
        \path (10.45,-2.4) edge[->] (11.55,-2.4)
              (1.4,-1.1) edge[->] (2.55,-1.1)
              (1.4,-2.4) edge[->] (2.55,-2.4)
              
              (1.4,-3.7) edge[->] (2.55,-3.7)
              (6.45, -2.4) edge[->] (7.75, -2.4);
        \path (6.45, -3.7) edge[->, hv path] (9.1,-2.9);
        \path (6.45, -1.1) edge[-, hv path] (10.85,-2.4);
        \draw[very thin] (1.4,-0.8) -- (-1.35,-0.8);
        \draw[very thin] (2.55,-0.8) -- (6.45,-0.8);

        \draw[very thin] (1.4,-0.8) -- (1.4,-1.6);
        \draw[very thin] (-1.35,-0.8) -- (-1.35,-1.6);
        \draw[very thin] (2.55,-0.8) -- (2.55,-1.6);
        \draw[very thin] (6.45,-0.8) -- (6.45,-1.6);
        
        \draw[very thin] (1.4,-1.6) -- (-1.35,-1.6);
        \draw[very thin] (2.55,-1.6) -- (6.45,-1.6);
        
         \draw[very thin] (1.4,-2) -- (-1.35,-2);
         \draw[very thin] (2.55,-2) -- (6.45,-2);
 
         \draw[very thin] (1.4,-2) -- (-1.35,-2);
         \draw[very thin] (1.4,-2) -- (1.4,-2.8);
         \draw[very thin] (-1.35,-2) -- (-1.35,-2.8);
         \draw[very thin] (2.55,-2) -- (6.45,-2);
         \draw[very thin] (2.55,-2) -- (2.55,-2.8);
         \draw[very thin] (6.45,-2) -- (6.45,-2.8);
         
        \draw[very thin] (1.4,-2.8) -- (-1.35,-2.8);
        \draw[very thin] (2.55,-2.8) -- (6.45,-2.8);
        
        \draw[very thin] (1.4,-3.2) -- (-1.35,-3.2);
        \draw[very thin] (2.55,-3.2) -- (6.45,-3.2);

        \draw[very thin] (1.4,-3.2) -- (-1.35,-3.2);
        \draw[very thin] (1.4,-3.2) -- (1.4,-4);
        \draw[very thin] (-1.35,-3.2) -- (-1.35,-4);
        \draw[very thin] (2.55,-3.2) -- (6.45,-3.2);
        \draw[very thin] (2.55,-3.2) -- (2.55,-4);
        \draw[very thin] (6.45,-3.2) -- (6.45,-4);
        
        \draw[very thin] (1.4,-4) -- (-1.35,-4);
        \draw[very thin] (2.55,-4) -- (6.45,-4);

        \draw[very thin] (7.75,-1.9) -- (7.75,-2.9);
        \draw[very thin] (10.45,-1.9) -- (10.45,-2.9);
        \draw[very thin] (7.75,-1.9) -- (10.45,-1.9);
        \draw[very thin] (7.75,-2.9) -- (10.45,-2.9);
        
        \draw[very thin] (11.55,-1.5) -- (11.55,-3.4);
        \draw[very thin] (13.55,-1.5) -- (13.55,-3.4);
        \draw[very thin] (11.55,-1.5) -- (13.55,-1.5);
        \draw[very thin] (11.55,-3.4) -- (13.55,-3.4);
        
\end{tikzpicture}
\caption{From the properties of (VIP) to the linear convergence of PG for solving (S-VIP)}\label{fig:3}
\end{center}
\end{figure}

\subsection{Sufficient conditions leading to  Assumption~\ref{assump1}}\label{app:suff-Lips-gap}
The following proposition provides the sufficient conditions for gradient Lipschitz of smooth gap function $g_{\lambda}(x)$ and $g_{\lambda}(x)$ is weakly convex on $\setX$.
\begin{proposition}\label{prop:suff-grad-Lips}
Let $g_{\lambda}(x)$ be defined by~\eqref{eq:SGF_1}. Then $y_{\lambda}(x)$ is single-valued also and the following assertions hold:

{\rm(i)} If $F$ and $\nabla F$ are continuous on compact set $\setX$, then $\nabla g_{\lambda}(x)$ is Lipschitz continuous on $\setX$.

{\rm(ii)} If $\nabla F$ and $F$ are Lipschitz continuous on $\setX$, then so is $\nabla g_{\lambda}(x)$.
\end{proposition}
\begin{proof}
{\rm(i)} First, by the boundedness of $\setX$ and the mean value theorem, $F$ is Lipschitz continuous on $\setX$.
Secondly, since projection mappings are non-expansive, we have $\|y_{\lambda}(x)-y_{\lambda}(x')\|\leq\|\left(x-\lambda F(x)\right)-\left(x'-\lambda F(x')\right)\|$, which implies the mapping $y$ is Lipschitz continuous on $\setX$.
Finally, it follows from the Lipschitz continuous of $F$ that $\|\nabla F\|$ is bounded on $\setX$, and from~\eqref{eq:grad_g} we obtain that $\nabla g_{\lambda}(x)$ is Lipschitz continuous on $\setX$.

{\rm(ii)} Similar as in the proof of statement (i), $\|\nabla F\|$ is bounded on $\setX$, and the mapping $x-y_{\lambda}(x)$ is Lipschitz continuous on $\setX$. Again, by~\eqref{eq:grad_g} it follows that $\nabla g_{\lambda}(x)$ is Lipschitz continuous on $\setX$.
\end{proof}

\subsection{Sufficient conditions leading to Assumption~\ref{assump2}}\label{app:suff-compact}
The following propositions provide some sufficient conditions on the compactness of $\setX_{\text{\tiny VIP}}^*$.
\begin{proposition}[Theorem 3.2-Theorem 3.4 of~\cite{HP90}]\label{prop:suff-compact-1} Let $\setX$ be a nonempty, closed, convex subset of $\RR^d$ and let $F$ be a continuous mapping from $\setX$ into $\RR^d$. Then the following statements hold:

{\rm(i)} If $F$ is coercive with respect to $\setX$, then the problem (VIP) has a nonempty, compact solution set $\setX_{\text{\tiny VIP}}^*$. 

{\rm(ii)} If there exists a vector $x^0\in\setX$ such that the set $\{x\in\setX:\langle F(x),x-x^0\rangle<0\}$ if nonempty, is bounded, then the problem (VIP) has a solution. Moreover, if the closed set $\{x\in\setX:\langle F(x),x-x^0\rangle\leq0\}$ is bounded, then the solution set $\setX_{\text{\tiny VIP}}^*$ is compact. 

{\rm(iii)} Suppose that $F$ is pseudo-monotone with respect to $\setX$ and that there exists a vector $x^0\in\setX$ such that $F(x^0)\in\interior(\setX^*)$, where $\interior(\cdot)$ denotes the interior of the set. Then, problem (VIP) has a nonempty, compact, convex solution set $\setX_{\text{\tiny VIP}}^*$.  
\end{proposition}
\subsection{Sufficient conditions leading to one of the three uniform error bound conditions \{(u-KL), (u-SEB), (u-HEB)\}}\label{subsec:suff_ebs}

The following theorem 
provides some sufficient conditions leading to one of the three uniform error bounds \{(u-KL), (u-SEB), (u-HEB)\}. 

\begin{theorem}\label{prop:suff-HEB}
Suppose Assumption~\ref{assump2} and one of the following conditions hold on (VIP):

{\rm(i)} $F$ is 
affine and $\setX$ is polyhedral;

{\rm(ii)} $F$ is Lipschitz continous on $\setX$ and there is $c_7>0$ such that $\dist\left(0,\partial g_{\lambda_1\lambda_2}(x)\right)>c_7\|y_{\lambda_1}(x)-y_{\lambda_2}(x)\|$ where $g_{\lambda_1\lambda_2}(x)=g_{\lambda_1}(x)-g_{\lambda_2}(x)$ with $\lambda_1>\lambda_2>0$ is the so called the D-gap function;

{\rm(iii)} $F$ is restrict strongly monotone on $\setX_{\text{\tiny VIP}}^*$, i.e., there exists $c_8>0$ such that
\[
\langle F(x)-F(\proj_{\setX_{\text{\tiny VIP}}^*}(x)),x-\proj_{\setX_{\text{\tiny VIP}}^*}(x)\rangle\geq c_8\dist^2\left(x,\setX_{\text{\tiny VIP}}^*\right),\quad\forall x\in\setX.
\]
Then, $\psi(x)$ satisfies \{\rm(u-KL),  (u-SEB), (u-HEB)\} property on $\Omega=\setX_{g_{\lambda}}^*=\setX_{\text{\tiny VIP}}^*$, with some $\mu>0$, $\delta>0$, and $\nu>0$.
\end{theorem}

\begin{proof}
{\rm(i)} Since $F$ is affine 
and $\setX$ is polyhedral, by~\eqref{eq:yx} and~\eqref{eq:grad_g}, we have that the mapping $\nabla g_{\lambda}(x)+\mathcal{N}_{\setX}(x)$ is piece-wise linear multi-functions. 
By 
a well-known result due to Robinson~\cite{Robinson1981} (or Theorem 3.3 of~\cite{ZN14}), we have the metric subregularity of the mapping $\nabla g_{\lambda}(x)$, i.e.,
\[
\dist(x,\overline{\setX}_{g_{\lambda}})\leq\mu_4\dist(0,\partial\psi(x)),\qquad\forall x\in\RR^d\quad\mbox{and}\quad\dist\left(0,\partial\psi(x)\right)\leq\varepsilon,
\]
with $\mu_4>0$ and $\varepsilon>0$. Let $\hat{x}=T_{\alpha}(x)$ with $f=g_{\lambda}$, $r=\mathcal{I}_{\setX}$, and $\alpha>0$. For any $x^*\in\setX_{g_{\lambda}}^*$ and any $x$ with $\dist(x,\setX_{g_{\lambda}}^*)<\delta$ and $x\in[\psi^*<\psi(x)<\psi^*+\nu]$ with $\nu\leq\frac{\frac{2}{\alpha}-L-\rho}{2\left(L+\frac{1}{\alpha}\right)^2}\varepsilon^2$, by statement (iii) of Proposition~\ref{prop}, we have 
$\psi^*\leq\psi(\hat{x})\leq\psi(x)$ and $\|x-\hat{x}\|\leq\sqrt{\frac{2}{\frac{2}{\alpha}-L-\rho}[\psi(x)-\psi(\hat{x})]}\leq\sqrt{\frac{2}{\frac{2}{\alpha}-L-\rho}\nu}\leq\frac{\varepsilon}{L+\frac{1}{\alpha}}$. Finally, combining with statement (vii) of Proposition~\ref{prop}, we have  
$\dist\left(0,\partial\psi(\hat{x})\right)\leq\left(L+\frac{1}{\alpha}\right)\|x-\hat{x}\|\leq\varepsilon$. Therefore, $\dist\left(\hat{x},\setX_{g_{\lambda}}^*\right)\leq\mu_4\dist(0,\partial\psi(\hat{x}))\leq\mu_4\left(L+\frac{1}{\alpha}\right)\|x-\hat{x}\|$. Since $\dist(\hat{x},\setX_{g_{\lambda}}^*)\geq\dist(x,\setX_{g_{\lambda}}^*)-\|x-\hat{x}\|$, again using the statement (iii) of Proposition~\ref{prop}, we have
\[
\begin{aligned}
\dist(x,\setX_{g_{\lambda}}^*)
\leq& \left[\mu_4\left(L+\frac{1}{\alpha}\right)+1\right]\|x-\hat{x}\| \\
\leq&\left[\mu_4\left(L+\frac{1}{\alpha}\right)+1\right]\sqrt{\frac{2}{\frac{2}{\alpha}-L-\rho}[\psi(x)-\psi(\hat{x})]}\\
\leq&\left[\mu_4\left(L+\frac{1}{\alpha}\right)+1\right]\sqrt{\frac{2}{\frac{2}{\alpha}-L-\rho}[\psi(x)-\psi^*]}.
\end{aligned}
\]

Combining Assumption~\ref{assump2} and $\setX_{g_{\lambda}}^*=[\psi\leq\psi^*]$, we conclude that {\rm(u-HEB)} holds on $\Omega=\setX_{g_{\lambda}}^*$.

{\rm(ii)} By Lemma 4.2 in~\cite{LMY22}, we have that the D-gap function $g_{\lambda_1\lambda_2}(x)$ satisfies the KL property at any solution of (VIP). By Assumption~\ref{assump2}, we have that $g_{\lambda_1\lambda_2}(x)$ satisfies {\rm(u-HEB)}. By the Theorem 3.2 of~\cite{YTF97} we obtain that $g_{\lambda_1\lambda_2}(x)$ and $g_{\lambda_1}(x)$ have the same solution set. Moreover, by the combination with the fact that $g_{\lambda_1\lambda_2}(x)\leq g_{\lambda_1}(x)$, we have that $g_{\lambda_1}(x)$ satisfies {\rm(u-HEB)} on the solution set.

{\rm(iii)} Since $F$ is restrict strongly monotone on $\setX_{\text{\tiny VIP}}^*=\setX_{g_{\lambda}}^*=[\psi\leq\psi^*]$,
we have that
\[
c_8\dist^2(x,[\psi\leq\psi^*])\leq \langle F(x),x-\proj_{\setX_{\text{\tiny VIP}}^*}(x)\rangle-\langle F(\proj_{\setX_{\text{\tiny VIP}}^*}(x)),x-\proj_{\setX_{\text{\tiny VIP}}^*}(x)\rangle,\quad\forall x\in\setX.
\]
By the fact that $\langle F(\proj_{\setX_{\text{\tiny VIP}}^*}(x)),x-\proj_{\setX_{\text{\tiny VIP}}^*}(x)\rangle\geq0$, we obtain 
\[
c_8\dist^2(x,[\psi\leq\psi^*])\leq \langle F(x),x-\proj_{\setX_{\text{\tiny VIP}}^*}(x)\rangle,\quad\forall x\in\setX.
\]
Adding the term $-\frac{1}{2\lambda}\|x-\proj_{\setX_{\text{\tiny VIP}}^*}(x)\|^2$ on both sides of above inequality, we have
\[
\left(c_8-\frac{1}{2\lambda}\right)\dist^2(x,[\psi\leq\psi^*])\leq \langle F(x),x-\proj_{\setX_{\text{\tiny VIP}}^*}(x)\rangle-\frac{1}{2\lambda}\|x-\proj_{\setX_{\text{\tiny VIP}}^*}(x)\|^2,\quad\forall x\in\setX.
\]
By the definition of $g_{\lambda}(x)$ and $g_{\lambda}(x^*)=0$ with $x^*\in\setX_{\text{\tiny VIP}}^*$, we have 
\[
\langle F(x),x-\proj_{\setX_{\text{\tiny VIP}}^*}(x)\rangle-\frac{1}{2\lambda}\|x-\proj_{\setX_{\text{\tiny VIP}}^*}(x)\|^2\leq g_{\lambda}(x)-g_{\lambda}(x^*),\quad\forall x\in\setX.
\]
If $\lambda>\frac{1}{2c_8}$, by Assumption~\ref{assump2}, we obtain the (u-HEB) with $\gamma=2$.
\end{proof}

The third condition in Theorem~\ref{prop:suff-HEB} (case (iii)) carries the name of restrict strongly monotonicity. Therefore, one may wonder how does that condition relate to monotonicity in general. The answer is that they are different. 

\begin{example}[Example of restrict strongly monotone mapping]\label{exp:rsm}
{\rm Consider the following VI problem
\begin{equation*}
\begin{aligned}
&\mbox{\rm Find $z^*\in\setZ$ such that}\\
&\langle F(z^*),z-z^*\rangle\geq0,\quad\forall z\in\setZ,
\end{aligned}
\end{equation*}
where $\setZ=\{z=(x,y)^{\top}\in\RR\times\RR_{++}
: 
x+y\leq10, y\geq1\}$ and $F(z)=(2xy^2,-2x^2y)^{\top}$. Obviously the solution of this problem is $z^*=(x^*,y^*)^{\top}$ with $x^*=0$, and $y^*\in[1,10]$. We can easily show that $F$ is non-monotone by setting $z=(1,1)^{\top}$, $z'=(1,2)^{\top}$. In that case, $\langle F(z)-F(z'),z-z'\rangle=\langle(2-8,-2+4)^{\top},(0,-1)^{\top}\rangle=-2<0$.
Next we are going to show that the problem also does not have a Minty solution; i.e.\ $\setS_{\text{\tiny MVI}}=\emptyset$. Note $\langle F(z),z-z^*\rangle=\langle(2xy^2,-2xy)^{\top},(x,y-y^*)^{\top}\rangle=2xy(xy-y+y^*)$.  

Case 1: $y^*\in[1,3]$. If $x=0.1$, $y=10$, $\langle F(z),z-z^*\rangle=2(y^*-9)<0$.

Case 2: $y^*\in(3,10]$. If $x=-1$, $y=1$, $\langle F(z),z-z^*\rangle=-2(y^*-2)<0$.

Combining these two cases, we conclude $\setS_{\text{\tiny MVI}}=\emptyset$.
Finally, we show that $F$ is restrict strongly monotone on $\setZ_{\text{\tiny VIP}}^*=\{z=(x,y)^{\top}\mid x=0, y\in[1,10]\}$. Letting $z_p^*=\proj_{\setZ_{\text{\tiny VIP}}^*}(z)$, we have 
\[
\langle F(z)-F(z_p^*),z-z_p^*\rangle=\langle (2xy^2,-2xy)^{\top},(x,y)^{\top}-(0,y)^{\top}\rangle=2x^2y^2\geq2x^2=2\|z-z_p^*\|^2.
\]
}
\end{example}

\begin{remark}
Some of the examples that we studied can be directly shown to satisfy the ($\ell$-SEB) condition;  e.g., Example~\ref{exp:nvip_1} satisfies the ($\ell$-SEB) with $\nu=1/8$ and $c_2=1$.
\end{remark}

Since $\setX$ is compact and one of the three conditions in Theorem~\ref{prop:suff-HEB} holds, we are led to the following conclusion:

\begin{corollary}
Consider a general VI model (\ref{Prob:VIP}) with a compact feasible set $\setX$.  Suppose that the non-monotone mapping $F$ satisfies one of the three conditions stipulated in Theorem~\ref{prop:suff-HEB}. Suppose we randomly uniformly select an initial solution from $\setX$, and apply the proximal gradient method (Algorithm~\ref{alg:PGOP}) on its smooth gap function defined as in (\ref{eq:SGF_1}). Then, with a positive probability the algorithm converges with a linear rate of convergence. 
\end{corollary}

\subsection{A homotopy continuation method}

If $F$ is 
affine and $\setX$ is polyhedral, 
it is possible to solve the model (\ref{Prob:VIP}) via homotopy continuation without requiring an initial solution solution that is close enough to the solution set. 
Take any simply linear and strongly monotone mapping, say $H(x)=x$, and let $0\le t \le 1$. Consider
\[
F^{(t)}(x) := t H(x) + (1-t) F(x),
\]
and 
\[
g^{(t)}_{\lambda}(x) := \max_{y\in\setX}\, \left[ \langle F^{(t)}(x),x-y\rangle-\frac{1}{2\lambda}\|x-y\|^2\right] .
\]
Choose any fixed parameter $0< \delta <1$.  Let $\overline{\setX}_{(t)}$ and $\setX_{(t)}^*$ be the set of critical points and optimal solutions of $\min_{x}\psi_{(t)}(x)=g_{\lambda}^{(t)}(x)+\mathcal{I}_{\setX}(x)$, and $\psi_{(t)}^*$ be its optimal value.

\begin{algorithm}[h]
\caption{Proximal Gradient Method on Gap Function with Homotopy Continuation}
{\bf Initialization:}  Let $t^{(0)}=1$. Find $x^0\in\setX$ as the solution to the strongly monotone model (\ref{Prob:VIP}) with $F=F^{(t^0)}$. 

\begin{algorithmic}[1]
\For{$k$}
\State $i:=1$. 

\State Let 
$t_i^{(k)}:=t^{(k)}(1-\delta^i)$.

\State 
Apply Algorithm \ref{alg:PGOP} to minimize $g^{(t_i^{(k)})}_{\lambda}$ over $\setX$ initialized from $x^k$. Let the solution be $x_{i}^k$. 

\State 
If $g^{(t_i^{(k)})}_{\lambda}(x_{i}^k)=0$ then let $x^{k+1}:=x_{i}^k$, $k:=k+1$ and Go To Line 1. 

\State 
If $g^{(t_i^{(k)})}_{\lambda}(x_{i}^k)>0$ then let $i:=i+1$ and Go To Line 3. 

\EndFor
\end{algorithmic}
\label{alg:hgap}
\end{algorithm}
In order to formally establish convergence of Algorithm~\ref{alg:hgap}, let us introduce the following assumption:
\begin{assumption}\label{assump4}
There is $\sigma>0$ such that $\min_{x\in\overline{\setX}_{(t)}\setminus\setX_{(t)}^*}\psi_{(t)}(x)-\psi_{(t)}^*>\sigma$ for any $t\in[0,1]$.
\end{assumption}
\begin{theorem}
    Suppose that Assumption~\ref{assump4} holds, $F$ 
    is affine linear, 
    and $\setX$ is a bounded polyhedron. Let $\{ x^0,x^1,\cdots\}$ be the sequence produced by Algorithm~\ref {alg:hgap}. Let $x^*$ be any cluster point of $\{ x^0,x^1,\cdots\}$. Then, $x^*$ is a solution for the VI model (\ref{Prob:VIP}). 
\end{theorem}
\begin{proof}
\textit{\underline{Step 1:} For any $i\in\setN$ and $k\in\setN$, ($\ell$-PEB) holds with $\nu^{(t_i^{(k)})}>0$.}

Since both $F$ and $H$ are 
affine, and the set $\setX$ is a polyhedron, by statement (i) of Theorem~\ref{prop:suff-HEB}, for any $i\in\setN$ and $k\in\setN$ there exists $\nu^{(t_i^{(k)})}>0$ such that the ($\ell$-PEB) holds on $\psi_{(t_i^{(k)})}=g^{(t_i^{(k)})}_{\lambda}+\mathcal{I}_{\setX}$. 

\textit{\underline{Step 2:} For any $k\in\setN$, as long as $i$ is sufficiently large, $x^k$ will fall within the set $[\psi_{(t_i^{(k)})}^*\leq\psi_{(t_i^{(k)})}\leq\psi_{(t_i^{(k)})}^*+\nu^{(t_i^{(k)})}]$.}

The left-hand side, $\psi_{(t_i^{(k)})}^*\leq\psi_{(t_i^{(k)})}(x^k)$, is straightforward. Next we are going to show $\psi_{(t_i^{(k)})}(x^k)\leq\psi_{(t_i^{(k)})}^*+\nu^{(t_i^{(k)})}$. Because

\[
\begin{aligned}
g^{(t_i^{(k)})}_{\lambda}(x^k)+\mathcal{I}_{\setX}(x^k)
=&g^{(t_i^{(k)})}_{\lambda}(x^k)\qquad\mbox{(line 4 of Algorithm~\ref{alg:hgap} guarantees $x^k\in\setX$)}\\
=&\max_{y\in\setX}\langle t_i^{(k)}H(x^k)+(1-t_i^{(k)})F(x^k),x^k-y\rangle-\frac{1}{2\lambda}\|x^k-y\|^2\\
\leq &g_{\lambda}^{(t^{(k)})}(x^k)+\max_{y\in\setX}\langle \delta^it^{(k)}(F(x^k)-H(x^k)),x^k-y\rangle\\
=&\max_{y\in\setX}\langle \delta^it^{(k)}(F(x^k)-H(x^k)),x^k-y\rangle,    
\end{aligned}
\]
by the boundedness of $\setX$ there must exist a positive number $D$ such that $g^{(t_i^{(k)})}_{\lambda}(x^k)+\mathcal{I}_{\setX}(x^k)\leq\delta^it^{(k)}D$. By $\delta\in(0,1)$, there is always a large enough $i$ such that $\psi_{(t_i^{(k)})}(x^k)\leq\psi_{(t_i^{(k)})}^*+\nu^{(t_i^{(k)})}$. 

\textit{\underline{Step 3:} For any $k\in\setN$, $i$ large enough, line 4 of Algorithm~\ref{alg:hgap} guarantees that $\psi_{(t_i^{(k)})}(x_i^k)=0$.}

By Theorem~\ref{theo:n-s-VIP} and Step 2, line 4 of Algorithm~\ref{alg:hgap} assures that $\psi_{(t_i^{(k)})}(x_i^k)=0$.

\textit{\underline{Step 4:} For each  $k\in\setN$, 
there exists a uniform upper bound $\overline{i}$ for $i$ exiting line 5 to line 1.}

By Assumption~\ref{assump4}, Proposition~\ref{prop:eb_set} and Step 1, we have for any $k$ and $i$, $\nu^{(t_i^{(k)})}$ is uniformly lower bounded by a positive number. Then, by the update rule in lines 2-6 of Algorithm~\ref{alg:hgap}, we have $i$ does not tend to infinity; instead, there exists a uniform upper bound $\overline{i}$ for $i$.

\textit{\underline{Step 5:} The sequence $\{t^{(k)}\}$ strictly decreases to $0$ with a finite number $\overline{k}\in\setN$ of steps.}

By Step 3 we have that $i$ does not tend to infinity; instead, there exists a uniform upper bound $\overline{i}$ for $i$. Then by the update rule of $t^{(k)}$, we have $\frac{t^{(k+1)}}{t^{(k)}}\leq 1-\delta^{\overline{i}}$. Thus, the sequence $t^k$ strictly decreases to $0$ in a finite number of steps.

\textit{\underline{Step 6:} $x^{\overline{k}+1}$ is the solution of the original problem~\eqref{Prob:VIP}.} 

By Step 2 and Step 5, we have $x^{k}\in[\psi_{(0)}^*\leq\psi_{(0)}\leq\psi_{(0)}^*+\nu^{(0)}]$, $\forall k\geq\overline{k}$. By Theorem~\ref{theo:n-s-VIP}, we obtain that $\forall k>\overline{k}$, $x^{k}$  will be the solution of the original problem~\eqref{Prob:VIP}; i.e., the cluster point of Algorithm~\ref{alg:hgap} is the solution of the original problem~\eqref{Prob:VIP}.
\end{proof}

\section{Numerical Experiments}\label{sec:NE}


\subsection{Examples in previous sections} {\ }

\textbf{Example~\ref{exp:nvip_1}.} For this example $y_{\lambda}(x)=\proj_{[-1,1]}(1+\lambda)x$ and $\nabla g_{\lambda}(x)=-x-[x-y_{\lambda}(x)]+\frac{1}{\lambda}[y_{\lambda}(x)-x]$. The numerical results are displayed in Figure~\ref{fig:num-exp1}. Given that the gap function of Example~\ref{exp:nvip_1} contains no critical points except the global optimum, this test case serves specifically to verify Algorithm~\ref{alg:PGOPg}'s performance. The left two plots of Figure~\ref{fig:num-exp1} illustrate the convergence of Algorithm~\ref{alg:PGOPg} to $g_{\lambda}=0$ in Example~\ref{exp:nvip_1} under different initialization points. The right two plots of Figure~\ref{fig:num-exp1} demonstrate that Algorithm~\ref{alg:PGOPg} converges to different solutions for Example~\ref{exp:nvip_1} when initialized from different starting points.

\begin{figure}[ht]
\begin{center}
{\includegraphics[width=0.3\textwidth]{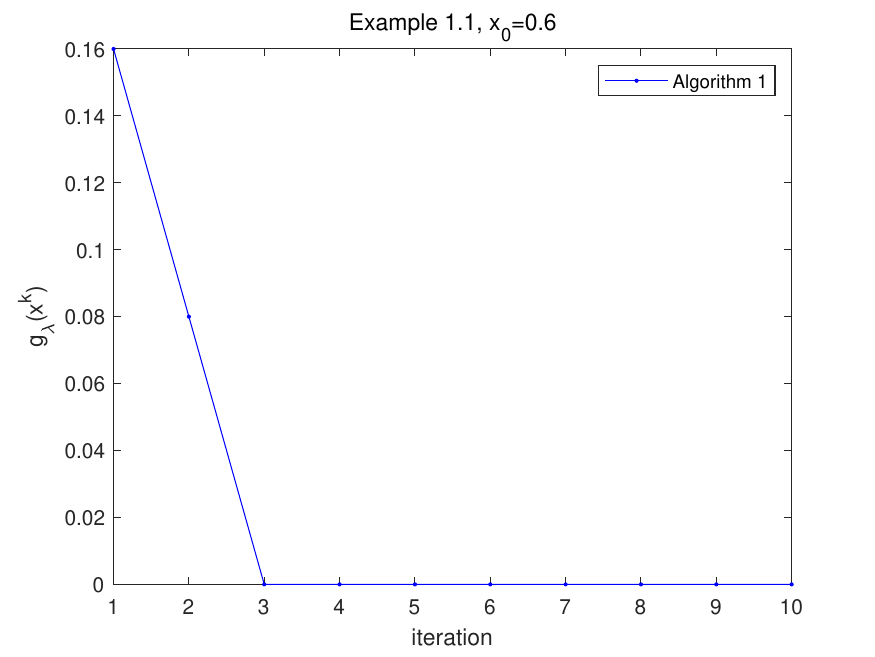}}
{\includegraphics[width=0.3\textwidth]{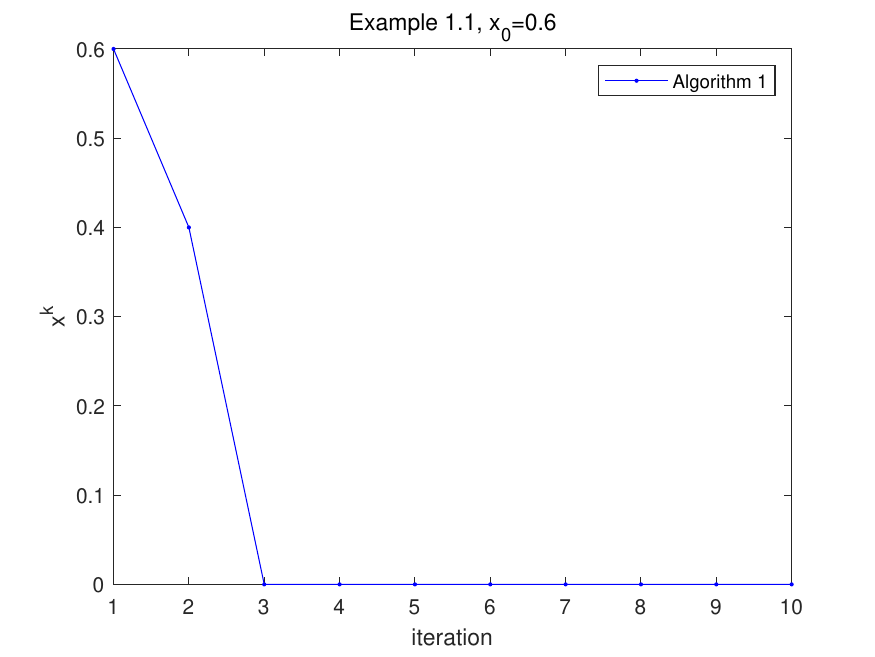}}\\
{\includegraphics[width=0.3\textwidth]{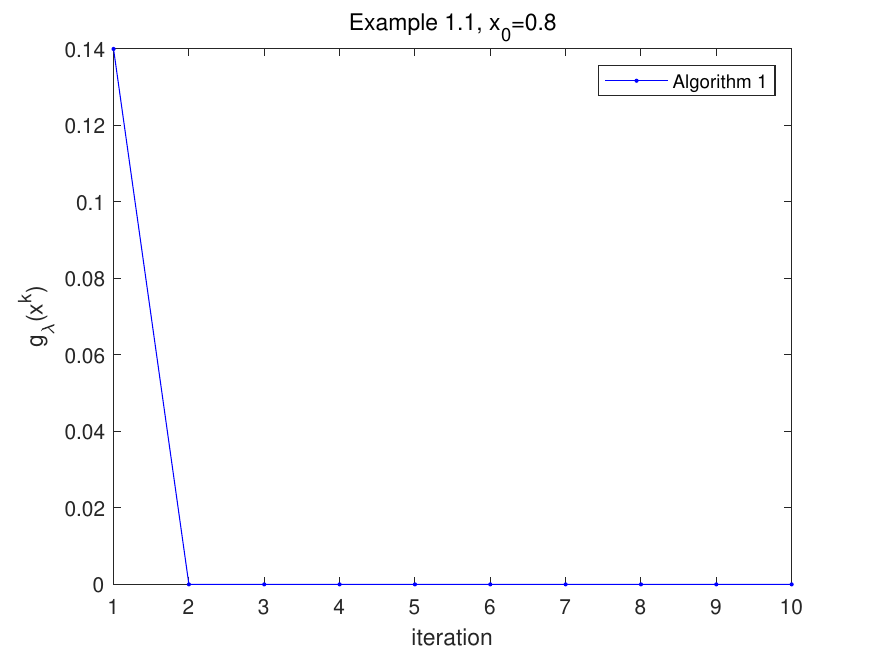}}
{\includegraphics[width=0.3\textwidth]{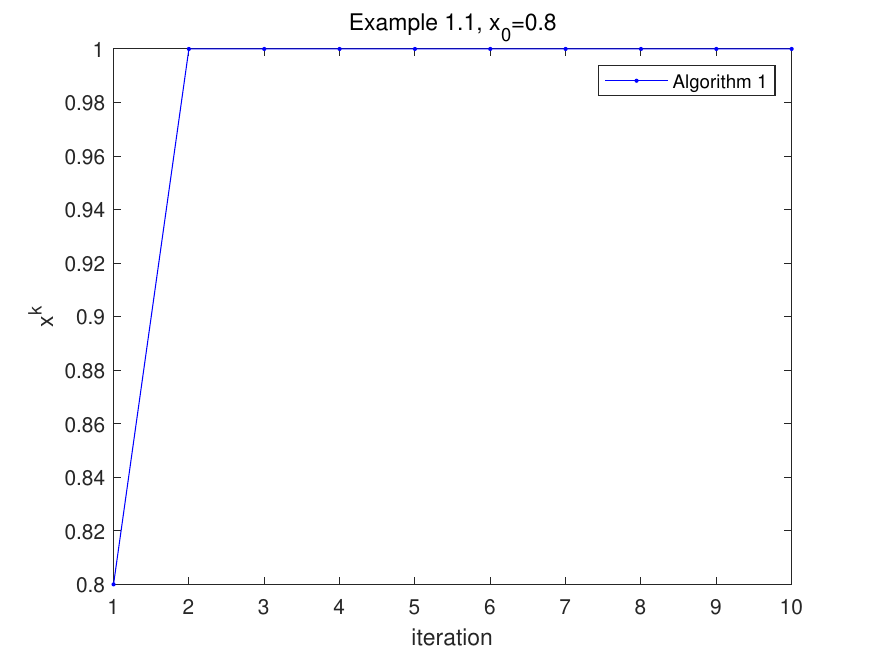}}
\vskip -0.1in
\caption{Left: Convergence of Algorithm~\ref{alg:PGOPg} on Example~\ref{exp:nvip_1}; Right: Solution generated by Algorithm~\ref{alg:PGOPg} on Example~\ref{exp:nvip_1}}\label{fig:num-exp1}
\end{center}
\end{figure}
\textbf{Example~\ref{exp:nvip_2}.} The numerical results are displayed in Figure~\ref{fig:exp2}. The left plot of Figure~\ref{fig:exp2} illustrates convergence of Algorithm~\ref{alg:PGOPg} to the solution of Example~\ref{exp:nvip_2} ($g_{\lambda}=0$).

\textbf{Example~\ref{exp:rsm}.} The right plot of Figure \ref{fig:exp2} illustrates convergence of Algorithm~\ref{alg:PGOPg} to the solution of Example~\ref{exp:rsm} ($g_{\lambda}=0$).
\begin{figure}[ht]
\begin{center}
{\includegraphics[width=0.3\textwidth]{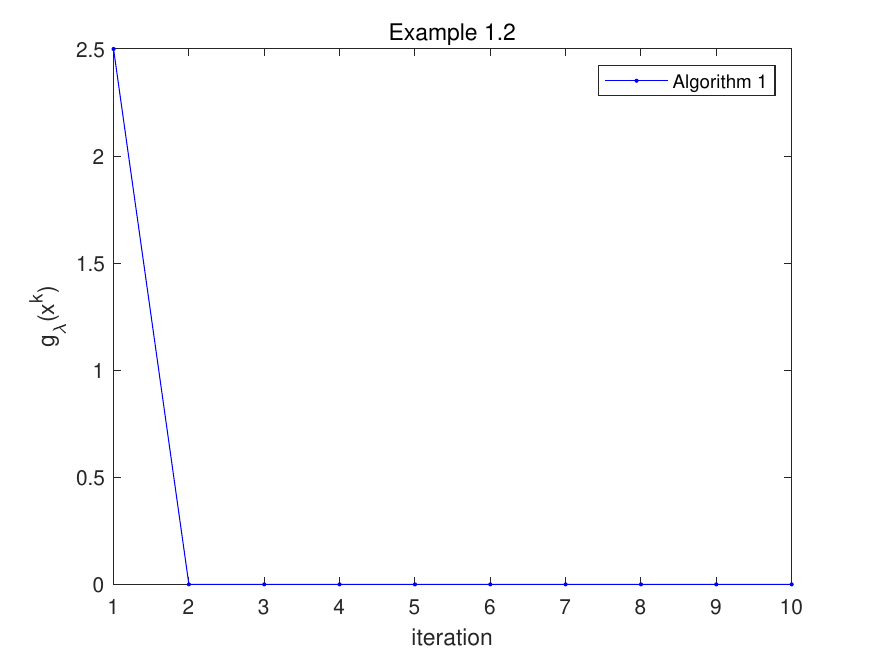}}
{\includegraphics[width=0.3\textwidth]{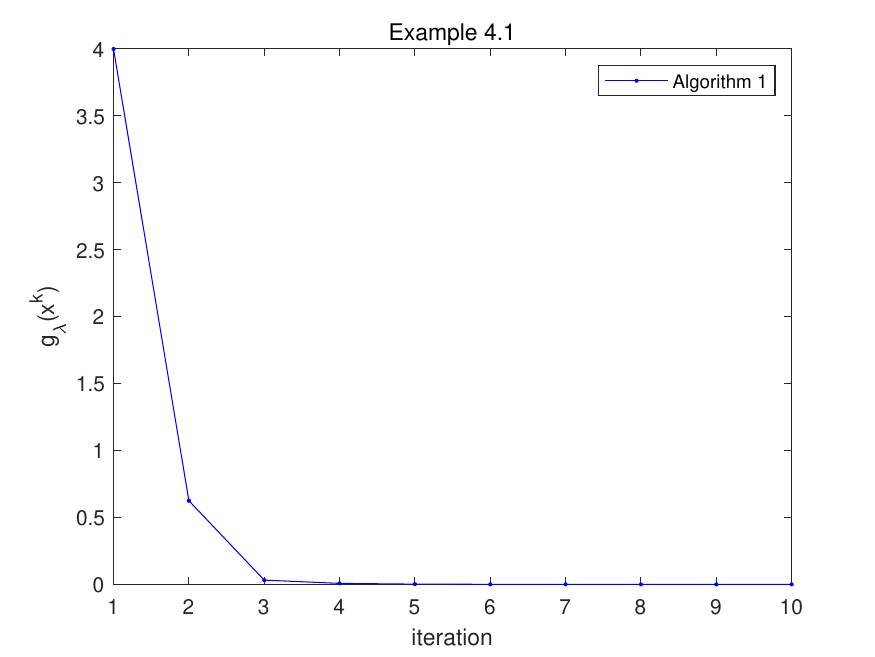}}
\vskip -0.1in
\caption{Convergence of Algorithm~\ref{alg:PGOPg} on Example~\ref{exp:nvip_2} (left) and Example~\ref{exp:rsm} (right)}\label{fig:exp2}
\end{center}
\end{figure}
\subsection{Bimatrix game}
In this subsection, we shall experiment with bimatrix game:
\begin{equation}\label{numeq:bi-m-game}
\max\limits_{x_1\in\Delta_1}\;x_1^{\top}Ax_2,\quad \max\limits_{x_2\in\Delta_2}\;x_1^{\top}Bx_2 
\end{equation}
where $x=(x_1^{\top},x_2^{\top})^{\top}$, $A, B\in\RR^{n_1\times n_2}$, and $\Delta_i=\{x_i\in\RR_+^{n_i} : \mathbf{1}_{n_i}^{\top}x_i-1=0\}$, $i= 1, 2$. By~\cite{FFP07}, model~\eqref{numeq:bi-m-game} can be written as the following (VIP):
\begin{equation}\label{numeq:VIP-bi-m-game}
\begin{aligned}
&\mbox{Find $x^*\in\Gamma$ such that} \\
&\langle F(x^*),x-x^*\rangle\geq0,\,\, \forall x=(x_1,x_2)^{\top}\in\Delta_1\times\Delta_2,   
\end{aligned}
\end{equation} 
where $F(x)=Mx$ with $M=\left(\begin{array}{cc}\mathbf{O}_{n_1\times n_1}&-A\\-B^{\top}&\mathbf{O}_{n_2\times n_2}\end{array}\right)$. In this case, $\nabla F=M^{\top}$, and the gradient of the gap function is
$\nabla g_{\lambda}(x)=Mx+M^{\top}[x-y_{\lambda}(x)]+\frac{1}{\lambda}[y_{\lambda}(x)-x]$, with $y_{\lambda}(x)=\proj_{\Delta_1\times\Delta_2}(x-\lambda Mx)$. 
The process of 
Algorithm~\ref{alg:PGOPg} for this problem is $x^{k+1}=\proj_{\Delta_1\times\Delta_2}(x^k-\alpha\nabla g_{\lambda}(x^k))$.

In our experiments, we test Algorithms~\ref{alg:PGOPg} and~\ref{alg:hgap}. We consider the example in equation (3.3) of~\cite{NRTV07}: 
$A=\left(
\begin{array}{ccc}
3 & 2& 0\\ 3 & 5& 6
\end{array}
\right)^{\top}$, 
and 
$B=\left(
\begin{array}{ccc}
3  & 2  & 3 \\ 2 & 6 & 1
\end{array}
\right)^{\top}$. 
It is easy to check that $x^*=(0,1/3,2/3,1/3,2/3)^{\top}$ is the solution to~\eqref{numeq:VIP-bi-m-game}. Therefore, we utilize $\|x^k-x^*\|$ to evaluate the performance of Algorithm~\ref{alg:PGOPg}. If we select the initial point $x^0 =(0,1/2,1/2,1/2,1/2)^{\top}$, Algorithm~\ref{alg:PGOPg} converges to the solution (see the left picture of Figure~\ref{fig:num-bi-m-game}). However, this does not guarantee that Algorithm~\ref{alg:PGOPg} converges to the solution of the problem in general cases. If we select the initial point $x^0 =(1/3,1/3,1/3,1/2,1/2)^{\top}$, which can also be used as an initial point of Algorithm~\ref{alg:hgap}, Algorithm~\ref{alg:PGOPg} only converges to the point $(0.3548,0.2742,0.3710,1/2,1/2)^{\top}$ which is not the solution of the problem~\eqref{numeq:VIP-bi-m-game}. However, Algorithm~\ref{alg:hgap} can generally converge to the solution of the problem~\eqref{numeq:VIP-bi-m-game}. The comparison between the convergence behavior of Algorithm 1 and Algorithm 3 is shown on the right of Figure~\ref{fig:num-bi-m-game}. 
\begin{figure}[ht]
\begin{center}
{\includegraphics[width=0.3\textwidth]{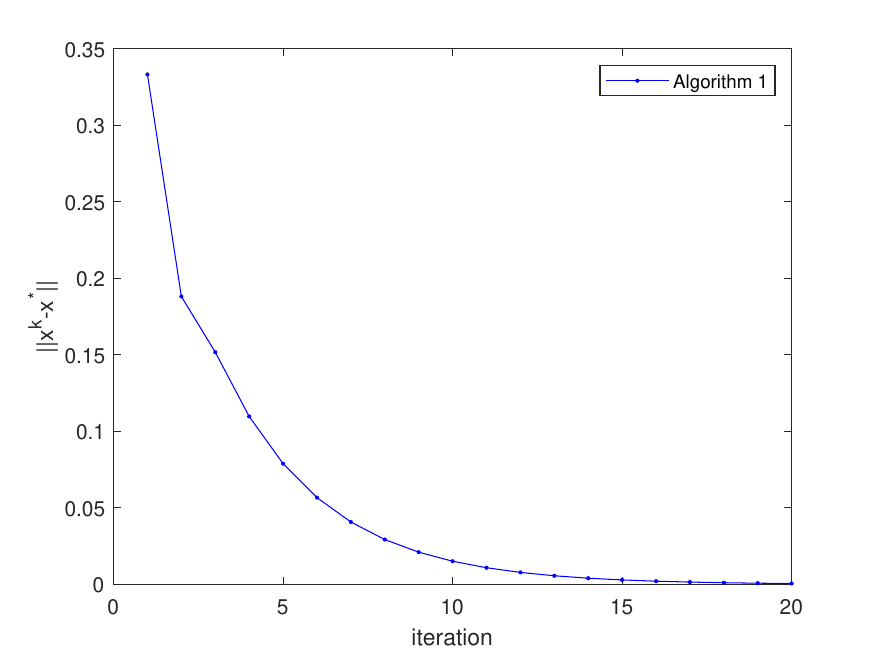}}
{\includegraphics[width=0.3\textwidth]{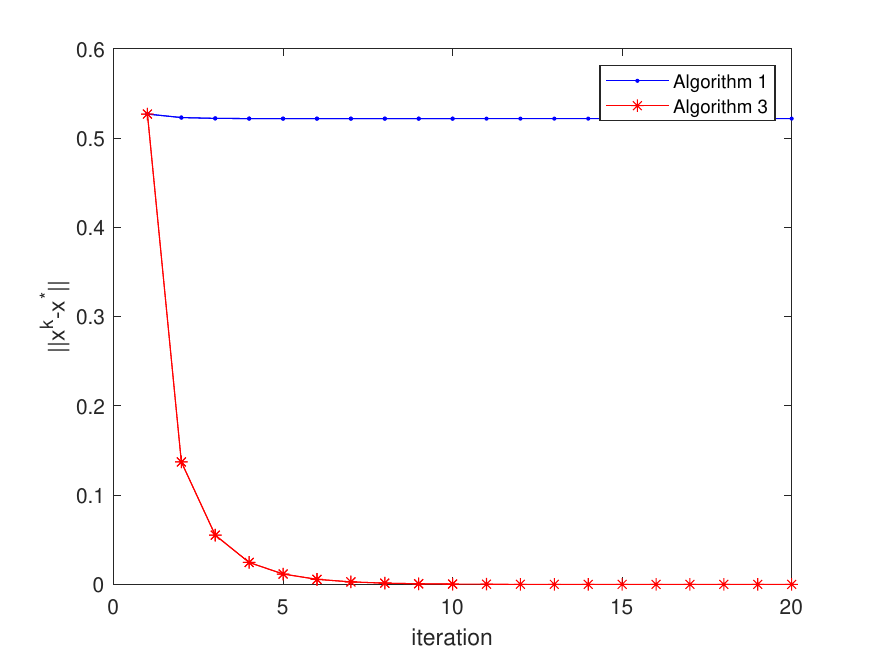}}
\vskip -0.1in
\caption{Left: Convergence of Algorithm~\ref{alg:PGOPg} with initial point $(0,1/2,1/2,1/2,1/2)^{\top}$; Right: Comparison of Algorithms~\ref{alg:PGOPg} and~\ref{alg:hgap} with initial point $(1/3,1/3,1/3,1/2,1/2)^{\top}$.}\label{fig:num-bi-m-game}
\end{center}
\end{figure}
In our experiences with more examples, we observe the same phenomenon. We randomly generate more test cases to compare the computational capabilities of Algorithm~\ref{alg:PGOPg} and Algorithm~\ref{alg:hgap}. We randomly generate date of matrices $A$ and $B$ with $n_1=\{3, 6, 9, 12, 15, 18\}$, $n_2=\{2, 4, 6, 8, 10, 12\}$, and each element of $A$ (or $B$) is randomly selected from the sets $\{0,1,...,10\}$, $\{0,1,...,30\}$, $\{0,1,...,50\}$, $\{0,1,...,70\}$, $\{0,1,...,90\}$, and $\{0,1,...,110\}$. We use the value of $g_{\lambda}(x^k)$ to measure the solution of problem~\eqref{numeq:VIP-bi-m-game}. Figure~\ref{fig:num-bi-m-game-2} illustrate computational results of Algorithms~\ref{alg:PGOPg} and~\ref{alg:hgap} on~\eqref{numeq:VIP-bi-m-game}, which shows that  Algorithm~\ref{alg:PGOPg} is not guaranteed to always converge to the global solution, but Algorithm~\ref{alg:hgap} does. 
\begin{figure}[ht]
\begin{center}
{\includegraphics[width=0.3\textwidth]{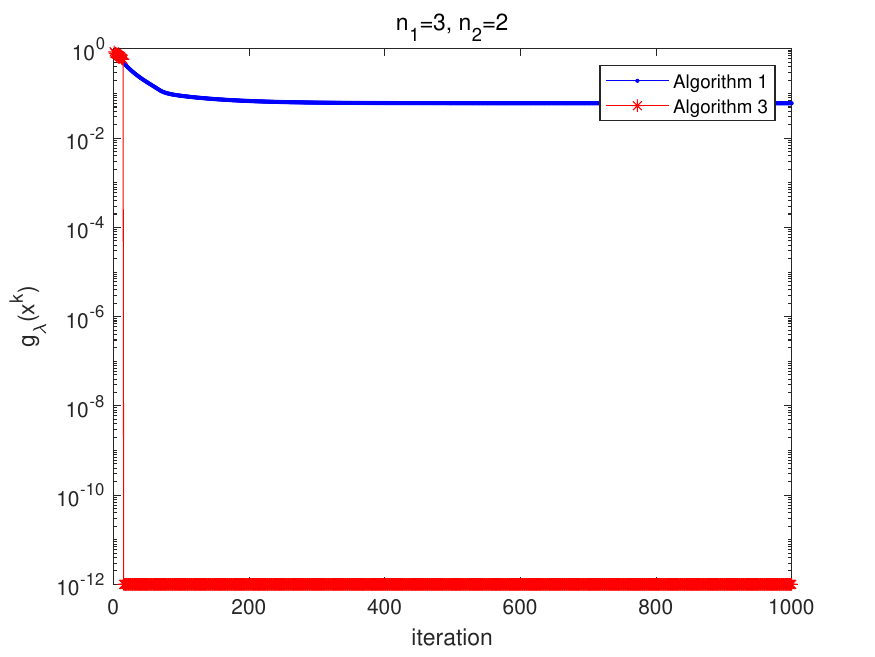}}
{\includegraphics[width=0.3\textwidth]{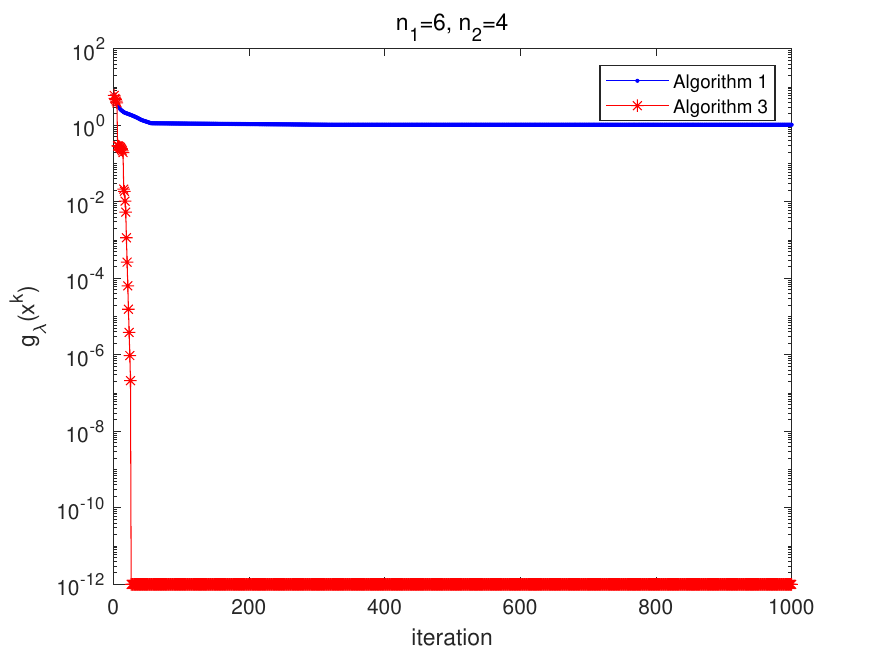}}
{\includegraphics[width=0.3\textwidth]{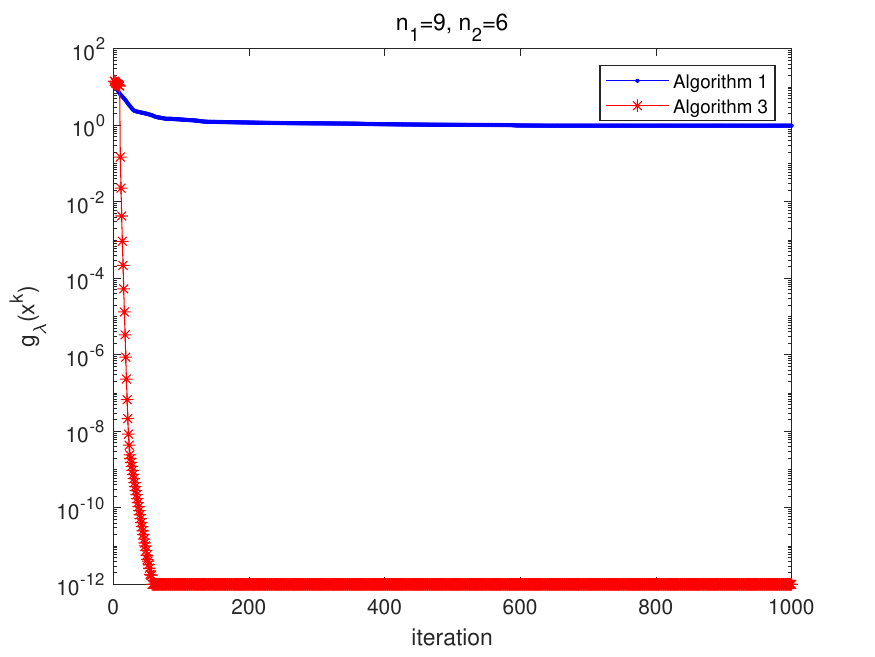}}\\
{\includegraphics[width=0.3\textwidth]{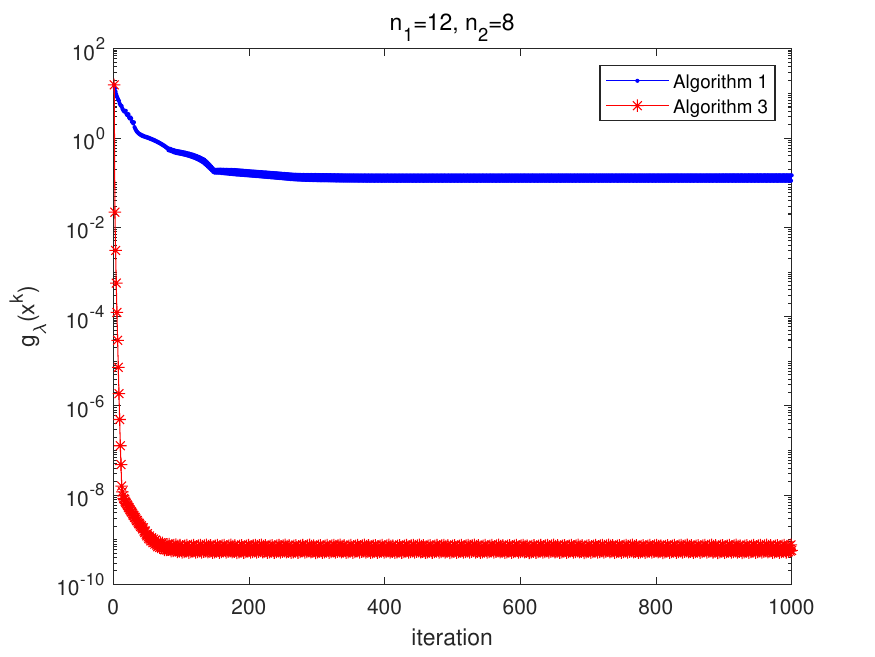}}
{\includegraphics[width=0.3\textwidth]{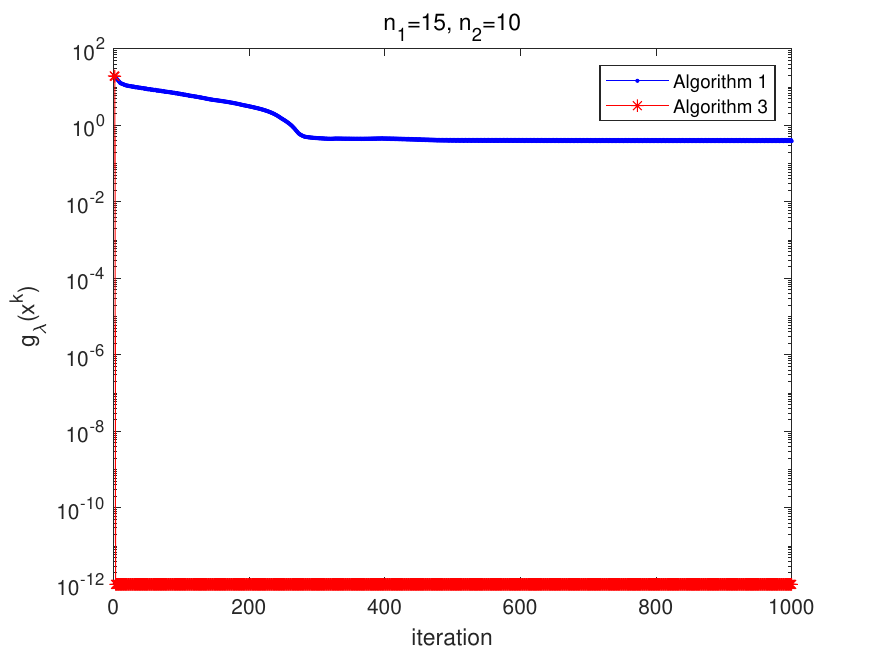}}
{\includegraphics[width=0.3\textwidth]{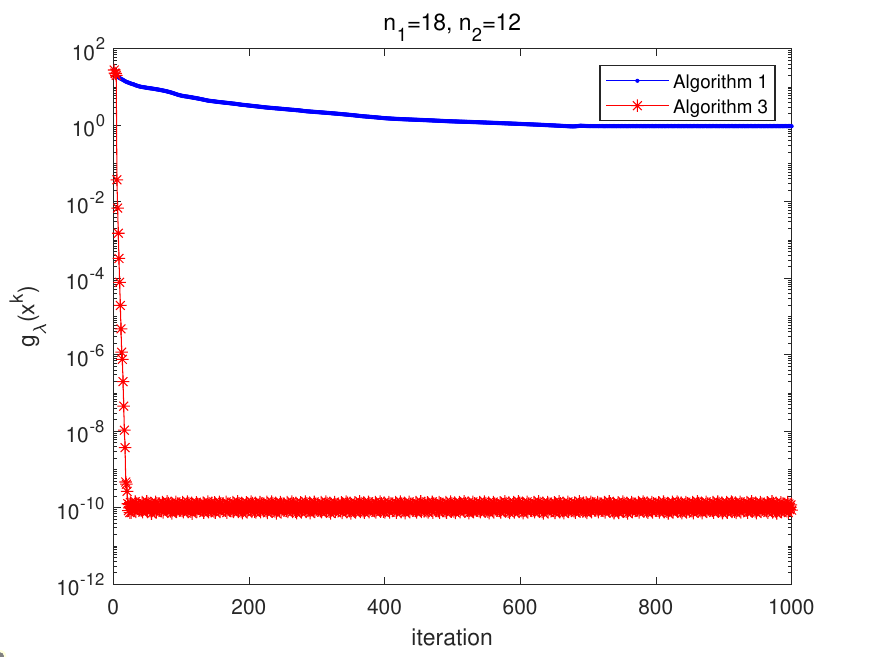}}

\vskip -0.1in
\caption{Comparison of Algorithm~\ref{alg:PGOPg} and~\ref{alg:hgap}.}\label{fig:num-bi-m-game-2}
\end{center}
\end{figure}

\subsection{Affine traffic equilibrium problem}

Traffic equilibrium is a way to describe how drivers choose their routes in a road network when every driver acts in accordance with his/her own self-interest. Instead of a central planner assigning routes, each driver picks the path that minimizes his/her own travel time (or cost). A user equilibrium means no driver can reduce his/her travel cost by unilaterally changing routes. The Traffic Equilibrium Problem (TEP) always considers a strongly connected network (a network is strongly connected if it contains at least one directed route from every node to every other node) $G(\mathcal{N},\mathcal{A})$, where $\mathcal{N}$ and $\mathcal{A}$ denote the sets of nodes and arcs, respectively.  Let $RS$ denote the set of origin-destination (O-D) pairs. For an O-D pair $rs\in RS$, let $q^{rs}$ be its traffic demand and $q=(q^1,...,q^{|RS|})^{\top}$ with $|RS|$ is the counts of O-D pairs. The independent variables are a set of path flows, denoted as $h_p^{rs}$, that must satisfy $\sum_{p\in P^{rs}}h_p^{rs}=q^{rs},\, \forall rs\in RS$, where $P^{rs}$ is a set of simple paths connecting $rs$. We denote $n=\sum\limits_{rs\in RS}|P^{rs}|$ to be the count of paths. Let $h\in\RR^{n}$ be the vector which is composed of $h_p^{rs}$. Then above equations can be rewritten as $Bh-q=0$, with $B\in\RR^{|RS|\times n}$. 

Furthermore, all path-flows are restricted to be non-negative to ensure a meaningful solution; that is, $h_p^{rs}\geq0,\, \forall rs\in RS, p\in P^{rs}$.
Let $f_a$ denote the traffic flow on link $a$. Then, the total flow on link $a$ is simply the sum of all paths using that link $f_a=\sum_{rs\in RS}\sum_{p\in P^{rs}}h_p^{rs}\delta_{pa}^{rs},\, \forall a\in\mathcal{A}$, where $\delta_{pa}^{rs}=1$, if link $a$ is on path $p$ connecting $rs$, and $0$ otherwise. Let $A\in\RR^{|\mathcal{A}|\times n}$ be the link-route incidence matrix which is composed of $\delta_{pa}^{rs}$.

The cost is often defined as a mapping of link flows~\cite{ND84}; that is $T=T(f),\quad\mbox{where}\quad f=(f_1,...,f_{|\mathcal{A}|})^{\top}$. The (TEP) can be formulated as a Variational Inequality Problem (VIP):
\[
\begin{array}{lc}
\mbox{(TEP):}  &\mbox{find non-negative vector $h^*$ such that}\\ &\langle A^{\top}T(Ah^*),h-h^*\rangle\geq0,\, \forall h\in\Gamma,    
\end{array}
\]
where $\Gamma=\left\{h : Bh-q=0\right\}$. If the mapping $T$ is affine and does not necessarily satisfy Assumption~1 of Reference~\cite{ND84}, it is generally difficult to ensure that (TEP) is monotone. However, Algorithm~\ref{alg:hgap} can be used to solve (TEP) in such cases. 

In this example, we use a larger network considered in~\cite{XBLC18} and \cite{YL24}. The network consists of $13$ nodes, $19$ links, and $4$ OD pairs (see Figure \ref{fig:NguyenDupuis}). The demand of each OD pair is provided as $q^{12}=400$, $q^{13}=800$, $q^{42}=600$, and $q^{43}=200$. The link travel time function is the asymmetric affine function mentioned in~\cite{ND84}: $T(f)=Cf+d$.

If $C$ is non-monotonic, then the mapping $A^{\top}T(Ah^*)$ may also be non-monotonic. Figure~\ref{fig:num-tep} illustrates the convergence of Algorithm~\ref{alg:PGOPg} and Algorithm~\ref{alg:hgap} under three non-monotonic selections of $C$ and $d$ (Left: $C=J_{|\mathcal{A}|}I_{|\mathcal{A}|}$, $d=\mathbf{1}_{|\mathcal{A}|}$; Middle: $C=J_{|\mathcal{A}|}I_{|\mathcal{A}|}(0.125,0.1,0.1,0.05,0.075,0.075,0.125,0.05,$ $0.125,0.125,0.05,0.05,0.025,0.05,0.1,0.025,0.1,0.1)^{\top}$, $d=(7, 9, 9, 12, 3, 9, 5, 13, 5, 9, 9, 10, 9, 6, 9, 8, 7, 14, 11)^{\top}$; Right: $C=J_{|\mathcal{A}|}X$, $d=y$), where $J_{|\mathcal{A}|}=\left(\begin{array}{ccc}
     0&\dots&1  \\
     \vdots&\iddots&\vdots\\
     1&\dots&0
\end{array}\right)_{|\mathcal{A}|\times|\mathcal{A}|}$, $X=Y+I_{|\mathcal{A}|}Z$ with $Y$ being a $|\mathcal{A}|\times|\mathcal{A}|$ random matrix with each element following a uniform distribution between $0$ and $0.1$, and $Z$ being a $|\mathcal{A}|\times|\mathcal{A}|$ random matrix with each element following a uniform distribution between $0$ and $1$, and $y$ is random vector with each element following a uniform distribution between $0$ and $10$.

\begin{figure}
\centering
\scriptsize
		\tikzstyle{format}=[rectangle,draw,thin,fill=white]
		\tikzstyle{test}=[diamond,aspect=2,draw,thin]
		\tikzstyle{point}=[coordinate,on grid,]
\begin{tikzpicture}
[
>=latex,
node distance=5mm,
 ract/.style={draw=blue!50, fill=blue!5,rectangle,minimum size=6mm, very thick, font=\itshape, align=center},
 racc/.style={rectangle, align=center},
 ractm/.style={draw=red!100, fill=red!5,rectangle,minimum size=6mm, very thick, font=\itshape, align=center},
 cirl/.style={draw, fill=yellow!20,circle,   minimum size=6mm, very thick, font=\itshape, align=center},
 raco/.style={draw=green!500, fill=green!5,rectangle,rounded corners=2mm,  minimum size=6mm, very thick, font=\itshape, align=center},
 hv path/.style={to path={-| (\tikztotarget)}},
 vh path/.style={to path={|- (\tikztotarget)}},
 skip loop/.style={to path={-- ++(0,#1) -| (\tikztotarget)}},
 vskip loop/.style={to path={-- ++(#1,0) |- (\tikztotarget)}}]
        \node (1) [cirl]{\baselineskip=3pt\small 1};
        \node (1a) [racc, above of = 1]{\baselineskip=3pt\footnotesize Origin};
        \node (1b) [racc, right of = 1, xshift=0,yshift=5]{\baselineskip=3pt\footnotesize 2};
        \node (1c) [racc, below of = 1, xshift=-4, yshift=-2]{\baselineskip=3pt\footnotesize 1};
        \node (12) [cirl, right of = 1, xshift=20]{\baselineskip=3pt\small 12};
        \node (12b) [racc, right of = 12, xshift=22, yshift=-15]{\baselineskip=3pt\footnotesize 18};
        \node (12c) [racc, below of = 12, xshift=-6, yshift=-2]{\baselineskip=3pt\footnotesize 17};
        \node (5) [cirl, below of = 1, yshift=-20]{\baselineskip=3pt\small 5};
        \node (5b) [racc, right of = 5, xshift=0, yshift=5]{\baselineskip=3pt\footnotesize 5};
        \node (5c) [racc, below of = 5, xshift=-4, yshift=-2]{\baselineskip=3pt\footnotesize 6};
        \node (9) [cirl, below of = 5, yshift=-20]{\baselineskip=3pt\small 9};
        \node (9b) [racc, right of = 9, xshift=0, yshift=5]{\baselineskip=3pt\footnotesize 12};
        \node (9c) [racc, right of = 9, xshift=-4, yshift=-20]{\baselineskip=3pt\footnotesize 13};
        \node (4) [cirl, left of = 5, xshift=-20]{\baselineskip=3pt\small 4};
        \node (4a) [racc, above of = 4]{\baselineskip=3pt\footnotesize Origin};
        \node (4b) [racc, right of = 4, xshift=0, yshift=5]{\baselineskip=3pt\footnotesize 3};
        \node (4c) [racc, right of = 4, xshift=-2, yshift=-20]{\baselineskip=3pt\footnotesize 4};
        \node (6) [cirl, right of = 5, xshift=20]{\baselineskip=3pt\small 6};
        \node (6b) [racc, right of = 6, xshift=0, yshift=5]{\baselineskip=3pt\footnotesize 7};
        \node (6c) [racc, below of = 6, xshift=-4, yshift=-2]{\baselineskip=3pt\footnotesize 8};
        \node (7) [cirl, right of = 6, xshift=20]{\baselineskip=3pt\small 7};
        \node (7b) [racc, right of = 7, xshift=0, yshift=5]{\baselineskip=3pt\footnotesize 9};
        \node (7c) [racc, below of = 7, xshift=-6, yshift=-2]{\baselineskip=3pt\footnotesize 10};
        \node (8) [cirl, right of = 7, xshift=20]{\baselineskip=3pt\small 8};
        \node (8c) [racc, below of = 8, xshift=-6, yshift=-2]{\baselineskip=3pt\footnotesize 11};
        \node (10) [cirl, below of = 6, yshift=-20]{\baselineskip=3pt\small 10};
        \node (10b) [racc, right of = 10, xshift=2, yshift=5]{\baselineskip=3pt\footnotesize 14};
        \node (11) [cirl, below of = 7, yshift=-20]{\baselineskip=3pt\small 11};
        \node (11b) [racc, right of = 11, xshift=2, yshift=5]{\baselineskip=3pt\footnotesize 15};
        \node (11c) [racc, below of = 11, xshift=-6, yshift=-2]{\baselineskip=3pt\footnotesize 16};
        \node (2) [cirl, below of = 8, yshift=-20]{\baselineskip=3pt\small 2};
        \node (2a) [racc, below of = 2]{\baselineskip=3pt\footnotesize Destination};
        \node (13) [cirl, below of = 10, yshift=-20]{\baselineskip=3pt\small 13};
        \node (13b) [racc, right of = 13, xshift=2, yshift=5]{\baselineskip=3pt\footnotesize 19};
        \node (3) [cirl, below of = 11, yshift=-20]{\baselineskip=3pt\small 3};
        \node (3a) [racc, below of = 3]{\baselineskip=3pt\footnotesize Destination};
        \path 
              (1) edge[->] (12)
              (1) edge[->] (5)
              (4) edge[->] (5)
              (5) edge[->] (6)
              (12) edge[->] (6)
              (6) edge[->] (7)
              (7) edge[->] (8)
              (12) edge[->] (8)
              (4) edge[->] (9)
              (5) edge[->] (9)
              (6) edge[->] (10)
              (7) edge[->] (11)
              (8) edge[->] (2)
              (9) edge[->] (10)
              (10) edge[->] (11)
              (11) edge[->] (2)
              (9) edge[->] (13)
              (11) edge[->] (3)
              (13) edge[->] (3);
\end{tikzpicture}
\caption{Nguyen and Dupius' Road Network (see~\cite{ND84})}\label{fig:NguyenDupuis}
\end{figure}
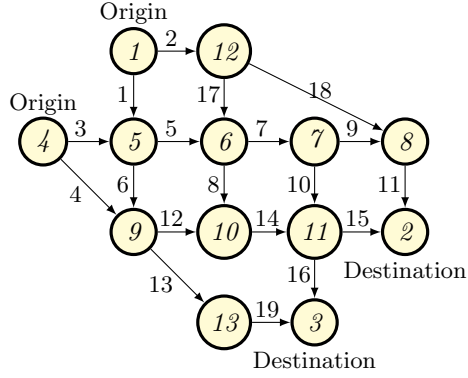
\begin{figure}
\centering
{\includegraphics[width=0.3\textwidth]{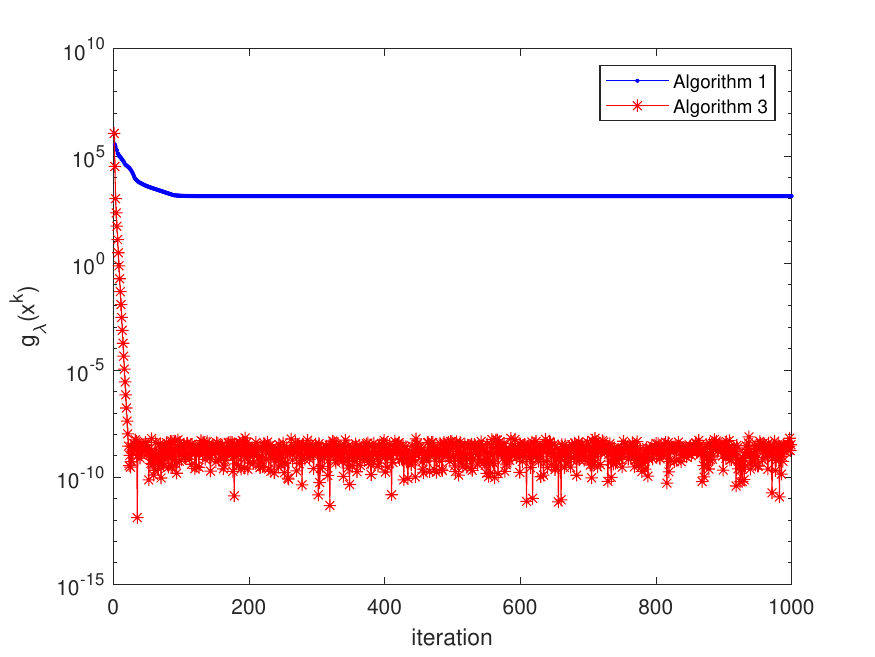}}
{\includegraphics[width=0.3\textwidth]{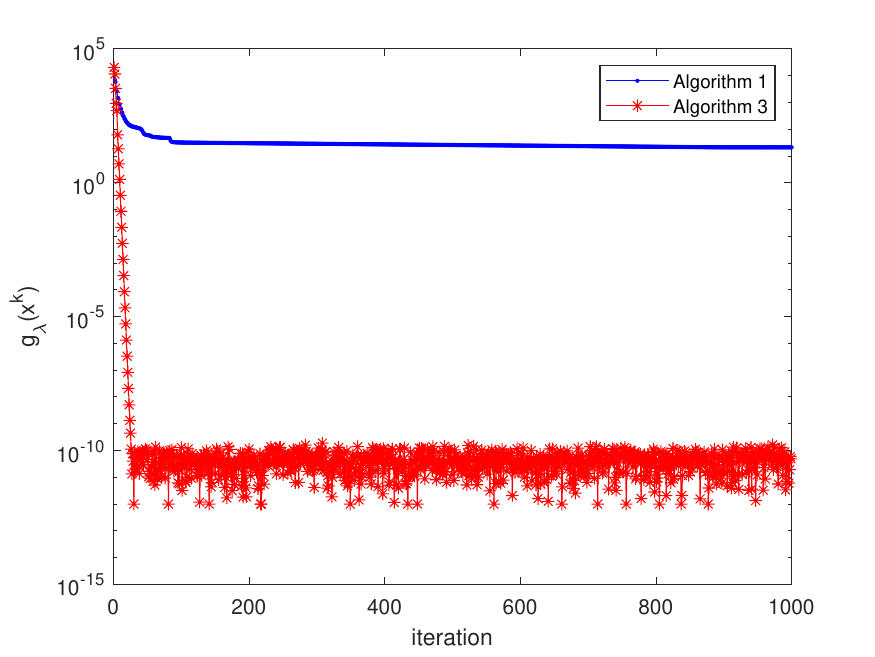}} 
{\includegraphics[width=0.3\textwidth]{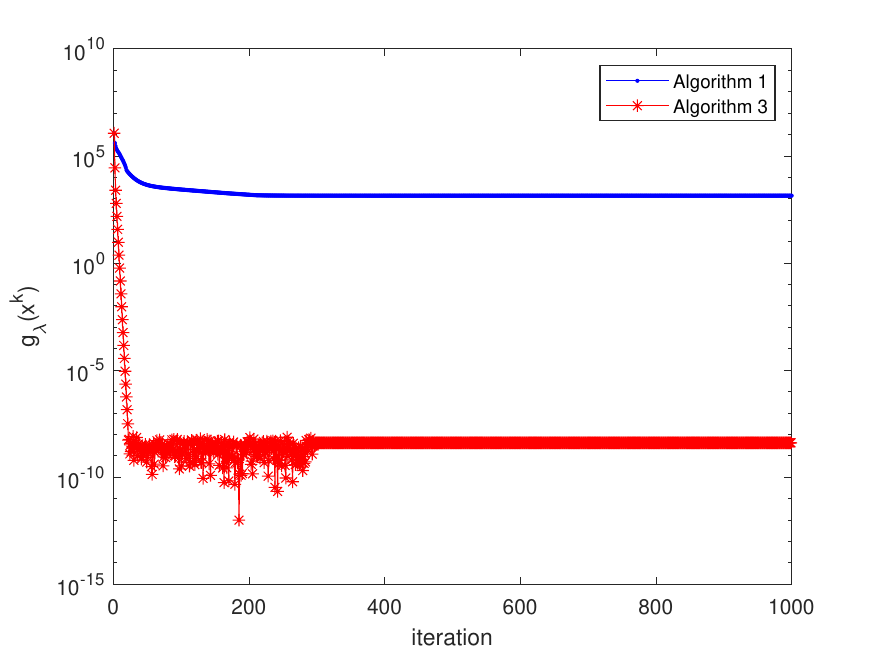}} 
\caption{Comparison of Algorithms~\ref{alg:PGOPg} and~\ref{alg:hgap} on Affine TEP.}\label{fig:num-tep}
\end{figure}

\subsection{Generative adversarial networks (GAN)}
In the original GAN paper~\cite{GAN14, GAN20}, the objective is formulated as a zero-sum game where the cost function of the discriminator $D_{\varphi}$ is given by the negative log-likelihood of the binary classification task between real or fake generated from $q_{\theta}$ by the generator,
\[
\min_{\theta}\max_{\varphi}\mathcal{L}(\theta,\varphi)\quad\mbox{where}\quad\mathcal{L}(\theta,\varphi):=-\EE_{x\sim p}\left[\log D_{\varphi}(x)\right]-\EE_{z\sim q_{\theta}}\left[\log(1-D_{\varphi}(z))\right].
\]
We take a toy GAN as an example:
\begin{equation}\label{eq:toy-GAN}
\min_{\theta\in\RR^{d_1}}\max_{\varphi\in\RR^{d_2}}\ \mathcal{L}(\theta,\varphi)=-\log(1+e^{-\varphi^\top\omega^*})-\log(1+e^{\varphi^{\top}\theta}),
\end{equation}
where $\omega^*\in\RR^{d_2}$. By~\cite{GBVVL18}, the variational inequality reformulation of the toy GAN is:
\begin{equation}\label{VIP_toy_GAN}
\begin{aligned}
&\mbox{Find $x^*\in\setX$ such that}\\
&\langle F(x^*),x-x^*\rangle\geq0,\quad\forall x\in\setX,
\end{aligned}
\end{equation}
with $x=(\theta,\phi)$, $\setX=\RR^{d_1}\times\RR^{d_2}$ and $F(x)=\left(\begin{aligned}\nabla_{\theta}\mathcal{L}(\theta,\varphi)\\ -\nabla_{\varphi}\mathcal{L}(\theta,\varphi)  
\end{aligned}\right)=\left(\begin{array}{c}-\frac{e^{\varphi^{\top}\theta}\varphi}{1+e^{\varphi^{\top}\theta}}\\-\frac{e^{-\varphi^{\top}\omega^*}\omega^*}{1+e^{-\varphi^{\top}\omega^*}}+\frac{e^{\varphi^{\top}\theta}\theta}{1+e^{\varphi^{\top}\theta}}\end{array}\right)$. 

If $d_1=d_2=1$, and $\omega^*=-2$, we can easily check that the solution is $(\theta^*,\varphi^*)=(-2,0)$. Additionally, let $x_1=(1,0.8)^{\top}$ and $x_2=(0.7,0.8)^{\top}$. We have $\langle F(x_1)-F(x_2),x_1-x_2\rangle=\langle(-0.0428,0.2445)^{\top},(0.3,0)^{\top}\rangle=-0.0128<0$. Thus the mapping $F(x)$ is non-monotone. Next, we are going to show that the solution set of Minty variational inequality (MVI, $\langle F(x),x-x^*\rangle\geq0$, $\forall x\in\setX$) of this problem is empty. Let $x=(3,1)^{\top}$, then $\langle F(x),x-x^*\rangle=\langle(-0.9526,
    4.6193)^{\top},(5,1)^{\top}\rangle=-0.1436<0$. With a carefully selected initial point, by using Algorithm~\ref{alg:PGOPg}, we can obtain the solution of this simple GAN.

Furthermore, for a generalized GAN, minimizing the gap function $g_{\lambda}(x)$ degenerates into the following optimization problem: $\min_{x\in\RR^d}\frac{\lambda}{2}\|F(x)\|^2$, which reproduces the consensus optimization (CO) training method of GANs (see \cite{MNG17}). The CO method can prevent the gradient descent-ascent (GDA) approach from circling around the initial point during training. However, the pure CO can converge to unhelpful stationary points of GAN ($\lambda\nabla F(x^k)^{\top}F(x^k)=0$ without $F=0$). Therefore, we train a GAN using the following consensus optimization scheme (modified Algorithm~\ref{alg:PGOPg}): $
x^{k+1}=x^k-\alpha[\lambda\nabla F(x^k)^{\top}F(x^k)-F(x^k)]$.

This algorithm scheme can also be regarded as another approach to help the algorithm escape the point where $\lambda\nabla F(x^k)^{\top}F(x^k)=0$ and move closer to the solution. We tested the modified Algorithm~\ref{alg:PGOPg} on a Deep Convolutional Generative Adversarial Network (DCGAN) using the MNIST dataset. The network architecture of the DCGAN used in the experiment is shown on the left of Figure~\ref{fig:num-gan}, while the output of the generator after 50 iterations of the algorithm is displayed on the right of Figure~\ref{fig:num-gan}.
\begin{figure}[ht]
\begin{center}
{\includegraphics[width=0.69\textwidth]{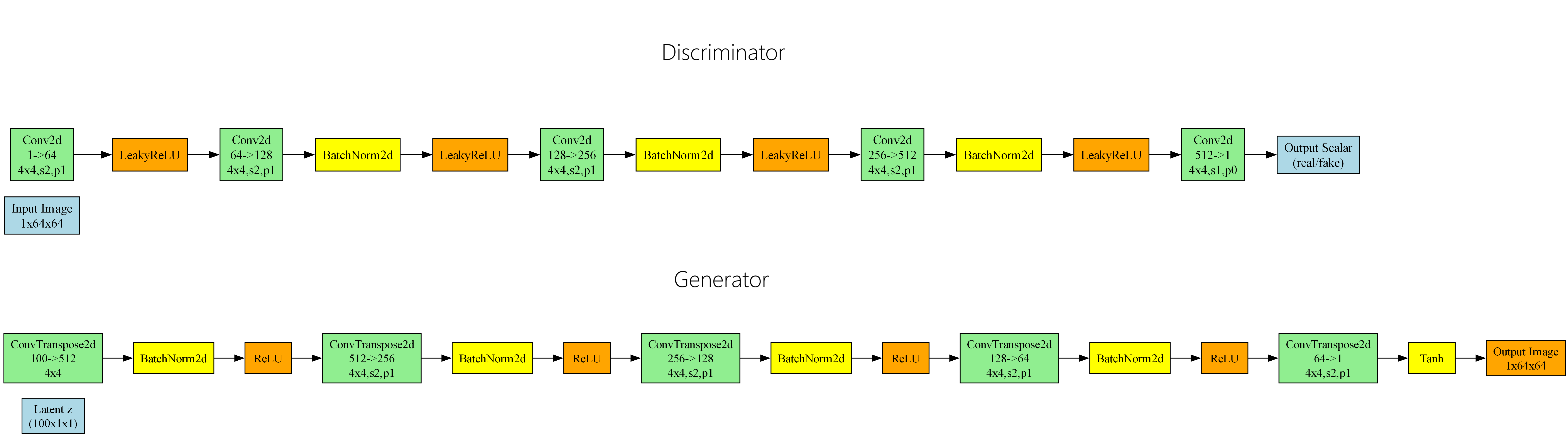}}
{\includegraphics[width=0.22\textwidth]{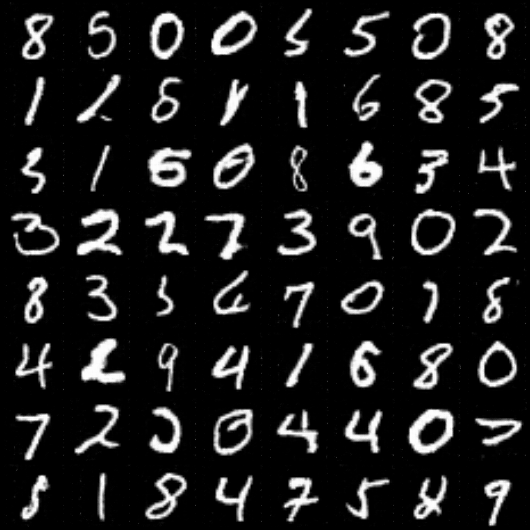}}
\vskip -0.1in
\caption{Modified Algorithm~\ref{alg:PGOPg} for GAN. Left: Network architecture of DCGAN; Right: Output of the generator after 50 iterations.}\label{fig:num-gan}
\end{center}
\end{figure}

{\bf Acknowledgments.}
L. Zhao and D. Zhu were partially supported by the National Key R\&D Program of China (Grant No.2023YFA0915202); the Major Project of the National Natural Science Foundation of China (Grant No.72293582); the Fundamental Research Funds for the Central Universities (the Interdisciplinary Program of Shanghai Jiao Tong University) (Grant No.YG2024QNA36). 

\appendix
\section{Proof of Proposition~\ref{prop}}\label{app:prop}
\begin{itemize}
\item[{\rm(i)}] Trivial.
\item[{\rm(ii)}] This follows immediately from the definitions $G_{\alpha}(x)$ and $E_{\alpha}(x)$. 
\item[{\rm(iii)}] Since $\nabla f$ is $L$-Lipschitz, one has
\begin{equation}\label{eq:E-phi}
\begin{aligned}
E_{\alpha}(x)=&f(x)+\langle\nabla f(x),T_{\alpha}(x)-x\rangle+r\left(T_{\alpha}(x)\right)+\frac{1}{2\alpha}\|x-T_{\alpha}(x)\|^2\\
      \geq&f\left(T_{\alpha}(x)\right)-\frac{L}{2}\|x-T_{\alpha}(x)\|^2+r\left(T_{\alpha}(x)\right)+\frac{1}{2\alpha}\|x-T_{\alpha}(x)\|^2\\
      =&\phi\left(T_{\alpha}(x)\right)+\frac{1}{2}\left(\frac{1}{\alpha}-L\right)\|x-T_{\alpha}(x)\|^2.
      \end{aligned}
\end{equation}
The optimality condition for the minimization problem in~\eqref{defi:Tk} 
gives 
\begin{equation}\label{eq:13}
0\in\nabla f(x)+\partial r\left(T_{\alpha}(x)\right)+\frac{1}{\alpha}\left(T_{\alpha}(x)-x\right).
\end{equation}
Since $r$ is l.s.c.\ and weakly convex with $\rho$, we have
\begin{equation}\label{eq:10}
\begin{aligned}
r(x)\geq&r\left(T_{\alpha}(x)\right)-\frac{\rho}{2}\|x-T_{\alpha}(x)\|^2-\langle\nabla f(x)+\frac{1}{\alpha}(T_{\alpha}(x)-x),x-T_{\alpha}(x)\rangle\\
\geq&r\left(T_{\alpha}(x)\right)-\frac{\rho}{2}\|x-T_{\alpha}(x)\|^2-\langle\nabla f(x),x-T_{\alpha}(x)\rangle+\frac{1}{\alpha}\|x-T_{\alpha}(x)\|^2.
\end{aligned}
\end{equation}
Adding $f(x)$ to both sides, by the definition of $E_{\alpha}(x)$, we have
\begin{equation}\label{eq:phi-E}
E_{\alpha}(x)\leq\phi(x)-\frac{1}{2}\left(\frac{1}{\alpha}-\rho\right)\|x-T_{\alpha}(x)\|^2.    
\end{equation}
Thus, by combining~\eqref{eq:E-phi} and~\eqref{eq:phi-E}, we obtain the desired result of this statement.
\item[{\rm(iv)}] By statement (i), $T_{\alpha}(x)$ is single-valued. From statements (ii) and~\eqref{eq:phi-E} we have
\[
G_{\alpha}(x)=\frac{1}{\alpha}\left(\phi(x)-E_{\alpha}(x)\right)\geq\frac{1}{2\alpha^2}\left(1-\alpha\rho\right)\|x-T_{\alpha}(x)\|^2.
\]
\item[{\rm(v)}] For $\alpha<\min\{1/L,1/\rho\}$, we have
\[
\alpha G_{\alpha}(x)=-\langle\nabla f(x), T_{\alpha}(x)-x\rangle-r\left(T_{\alpha}(x)\right)+r(x)-\frac{1}{2\alpha}\|x-T_{\alpha}(x)\|^2.
\]
Let $\nu\in\partial r(x)$. Thanks to the weakly-convex of $r$ we get
\begin{equation}
\begin{aligned}
\alpha G_{\alpha}(x)\leq&-\langle\nabla f(x), T_{\alpha}(x)-x\rangle-\langle\nu, T_{\alpha}(x)-x\rangle+\frac{\rho}{2}\|x-T_{\alpha}(x)\|^2-\frac{1}{2\alpha}\|x-T_{\alpha}(x)\|^2\nonumber\\
=&-\langle\nabla f(x)+\nu, T_{\alpha}(x)-x\rangle-\frac{1}{2}\left(\frac{1}{\alpha}-\rho\right)\|x-T_{\alpha}(x)\|^2\nonumber\\
\leq&\|\nabla f(x)+\nu\|\cdot\|x-T_{\alpha}(x)\|-\frac{1}{2}\left(\frac{1}{\alpha}-\rho\right)\|x-T_{\alpha}(x)\|^2\nonumber\\
\leq&\frac{\alpha}{2(1-\alpha\rho)}\|\nabla f(x)+\nu\|^2.
\end{aligned}
\end{equation}
Therefore, $G_{\alpha}(x)\leq\frac{1}{2(1-\alpha\rho)}\|\nabla f(x)+\nu\|^2$, $\forall\nu\in\partial r(x)$, and the claim is proven.
\item[{\rm(vi)}] This statement is a simple consequence of (iv) and (v). 
\item[{\rm(vii)}] By the optimality condition of the optimization problem defining $T_{\alpha}(x)$ we have
\[
0\in\nabla f(x)+\partial r(T_{\alpha}(x))+\frac{1}{\alpha}[T_{\alpha}(x)-x].
\]
It follows that $\nabla f(T_{\alpha}(x))-\nabla f(x)+\frac{1}{\alpha}[x-T_{\alpha}(x)]\in\partial\phi(T_{\alpha}(x))$. By the $L$-gradient Lipschitz property of $f$, we have $\dist\left(0,\partial\phi(T_{\alpha}(x))\right)\leq\left(L+\frac{1}{\alpha}\right)\|x-T_{\alpha}(x)\|$.
\end{itemize}
\section{Proof of Lemma~\ref{lemma:1}}\label{app:lem}
Denote $\Delta=\langle\nabla f(x),T_{\alpha}(x)-u\rangle+r\left(T_{\alpha}(x)\right)-r(u)$. First, we estimate the lower bound of $\Delta$:\\
\begin{equation}\label{eq:bound1-f-n}
\begin{aligned}
\Delta=&\langle\nabla f(x),T_{\alpha}(x)-u\rangle+r\left(T_{\alpha}(x)\right)-r(u)\nonumber\\
=&\langle\nabla f(x),T_{\alpha}(x)-x\rangle+\langle\nabla f(x),x-u\rangle+r\left(T_{\alpha}(x)\right)-r(u)\nonumber\\
\geq&f\left(T_{\alpha}(x)\right)-f(x)-\frac{L}{2}\|x-T_{\alpha}(x)\|^2+\langle\nabla f(x),x-u\rangle+r\left(T_{\alpha}(x)\right)-r(u)\nonumber\\
&\qquad\qquad\qquad\qquad\qquad\qquad\mbox{(since $f$ is gradient Lipschitz with modulus $L$)}\nonumber\\
=&\phi\left(T_{\alpha}(x)\right)-\phi(u)-\frac{L}{2}\|x-T_{\alpha}(x)\|^2+\underbrace{f(u)-f(x)-\langle\nabla f(x),u-x\rangle}_{\delta_1}.
\end{aligned}
\end{equation}
By the gradient $L$-Lipschitz continuity of $f$, we know that $f$ is also $L$-weakly convex (see~\cite{nesterov03} Lemma 1.2.3). Next, we estimate the term $\delta_1$ in~\eqref{eq:bound1-f-n}: $\delta_1=f(u)-f(x)-\langle\nabla f(u),u-x\rangle\geq-\frac{L}{2}\|x-u\|^2$. Therefore, we have
\begin{equation}\label{eq:bound1}
\Delta\geq\phi\left(T_{\alpha}(x)\right)-\phi(u)-\frac{L}{2}\|x-T_{\alpha}(x)\|^2-\frac{L}{2}\|x-u\|^2.
\end{equation}
Since $\alpha<\frac{1}{\rho}$, the minimization problem of~\eqref{defi:Tk} is $(\frac{1}{\alpha}-\rho)$-strongly convex. We have
\[
\begin{aligned}
&\langle\nabla f(x),u-T_{\alpha}(x)\rangle+r(u)-r\left(T_{\alpha}(x)\right)+\frac{1}{2\alpha}\|x-u\|^2-\frac{1}{2\alpha}\|x-T_{\alpha}(x)\|^2\\
\geq&\frac{1}{2}\left(\frac{1}{\alpha}-\rho\right)\|u-T_{\alpha}(x)\|^2.
\end{aligned}
\]
Therefore, 
\begin{equation}\label{eq:optimalcondition-gn1}
\begin{aligned}
\Delta=&\langle\nabla f(x),T_{\alpha}(x)-u\rangle+r\left(T_{\alpha}(x)\right)-r(u)\nonumber\\
\leq&\frac{1}{2\alpha}\|x-u\|^2-\frac{1}{2\alpha}\|x-T_{\alpha}(x)\|^2-\frac{1}{2}\left(\frac{1}{\alpha}-\rho\right)\|u-T_{\alpha}(x)\|^2.
\end{aligned}
\end{equation}
The desired result follows by
combing~\eqref{eq:bound1} with~\eqref{eq:optimalcondition-gn1}.

\section{Proof of Theorem~\ref{theo:n-s}}\label{app:n-s}
{\rm(i)} By taking $x=x^k$ and $u=x$, and noting $T_{\alpha}(x^k)=x^{k+1}$ and Lemma~\ref{lemma:1}, we have
\begin{equation}\label{eq:VI_linear_2}
\frac{1}{2}\left(\frac{1}{\alpha}-\rho\right)\|x-x^{k+1}\|^2\leq\phi(x)-\phi(x^{k+1})+\frac{1}{2}\left(\frac{1}{\alpha}+L\right)\|x-x^k\|^2-\frac{1}{2}\left(\frac{1}{\alpha}-L\right)\|x^k-x^{k+1}\|^2.
\end{equation}
Since $\{\phi(x^k)\}$ is strictly decreasing and $\phi(x^k)\geq\phi^*$. Note that the equality implies $x^k\in[\phi\leq\phi^*]$. Therefore, for given $x^0\in[\phi^*<\phi<\phi^*+\nu]$, we have $x^k\in[\phi^*<\phi<\phi^*+\nu]$. Let $x_k^*=\arg\min_{x\in[\phi\leq\phi^*]}\|x-x^k\|$, then $\phi(x_k^*)=\phi^*$. By taking $x=x_k^*$ in~\eqref{eq:VI_linear_2} and $\alpha<\frac{1}{L+\rho}$ and $\dist(x^{k+1},[\phi\leq\phi^*])\leq\|x_k^*-x^{k+1}\|$, we have 
\begin{equation}\label{eq:VI_linear_3}
\begin{aligned}
&\;\;\quad\quad\frac{1}{2}\left(\frac{1}{\alpha}-\rho\right)\dist^2(x^{k+1},[\phi\leq\phi^*])
\leq\frac{1}{2}\left(\frac{1}{\alpha}-\rho\right)\|x_k^*-x^{k+1}\|^2\\
&\;\quad\overset{\eqref{eq:VI_linear_2}}{\leq}\phi^*-\phi(x^{k+1})+\frac{1}{2}\left(\frac{1}{\alpha}+L\right)\dist^2(x^k,[\phi\leq\phi^*])-\frac{1}{2}\left(\frac{1}{\alpha}-L\right)\|x^k-x^{k+1}\|^2\\
&\overset{\phi(x^{k+1})\geq\phi^*}{\leq}\frac{1}{2}\left(\frac{1}{\alpha}+L\right)\dist^2(x^k,[\phi\leq\phi^*])-\frac{1}{2}\left(\frac{1}{\alpha}-L\right)\|x^k-x^{k+1}\|^2.
\end{aligned}
\end{equation}
This, together with ($\ell$-PEB), yields
\[
\begin{aligned} 
    & \frac{1}{2}\left(\frac{1}{\alpha}-\rho\right)\dist^2(x^{k+1},[\phi\leq\phi^*]) \\ &\leq\frac{1}{2}\left(\frac{1}{\alpha}+L\right)\dist^2(x^k,[\phi\leq\phi^*])
    -\frac{1}{2c_1^2}\left(\frac{1}{\alpha}-L\right)\dist^2(x^k,[\phi\leq\phi^*]).
    \end{aligned}
    \]
Since $\alpha\leq\frac{1}{(c_1^2+1)(L+\rho)}$, we have $\frac{1}{c_1^2}\frac{1}{\alpha}-\left(1+\frac{1}{c_1^2}\right)L>\rho$ and $\frac{\left(1-\frac{1}{c_1^2}\right)\frac{1}{\alpha}+\left(1+\frac{1}{c_1^2}\right)L}{\frac{1}{\alpha}-\rho}=\frac{\frac{1}{\alpha}-\left[\frac{1}{c_1^2}\frac{1}{\alpha}-\left(1+\frac{1}{c_1^2}\right)L\right]}{\frac{1}{\alpha}-\rho}\in(0,1)$. Thus, $\dist^2(x^{k+1},[\phi\leq\phi^*])\leq\beta\dist^2(x^k,[\phi\leq\phi^*])$, where $\beta=\frac{\left(1-\frac{1}{c_1^2}\right)\frac{1}{\alpha}+\left(1+\frac{1}{c_1^2}\right)L}{\frac{1}{\alpha}-\rho}\in(0,1)$.

(ii) By the weakly convexity of $r$ and $\alpha<\min\{1/L,1/\rho\}$, it follows that $T_{\alpha}(x)$ is single-valued. Let $T_{\alpha}(x)_p=\arg\min_{y\in[\phi\leq\phi^*]}\|y-T_{\alpha}(x)\|$.  For any $x\in[\phi^*<\phi<\phi^*+\nu]$, observe that 
\begin{equation}\label{TD}
\begin{aligned}
    & \dist\left(x, [\phi\leq\phi^*]\right)\leq \|x-T_{\alpha}(x)_p\|
    \leq \|T_{\alpha}(x)-T_{\alpha}(x)_p\|+\|x-T_{\alpha}(x)\|\nonumber\\
    =&\dist\left(T_{\alpha}(x),[\phi\leq\phi^*]\right)+\|x-T_{\alpha}(x)\| 
    \leq \beta \dist\left(x,[\phi\leq\phi^*]\right)+\|x-T_{\alpha}(x)\|.\nonumber
    \end{aligned}
    \end{equation}
Therefore, $\dist\left(x,[\phi\leq\phi^*]\right)\leq\frac{1}{1-\beta}\|x-T_{\alpha}(x)\|$, which shows that the level-set error bound ($\ell$-PEB) 
holds on $[\phi^*<\phi<\phi^*+\nu]$. Additionally, if $x\in[\phi\leq\phi^*]$, $ \dist\left(x, [\phi\leq\phi^*]\right)=0$, then the level-set error bound ($\ell$-PEB) also holds. 
Therefore, the level-set error bound ($\ell$-PEB) 
holds on $[\phi^*\leq\phi<\phi^*+\nu]$.
\section{Proof of Proposition~\ref{prop:beb}}\label{app:rela_leb}
\begin{itemize}
\item[{\rm(i)}] {\bf ($\ell$-SEB)$\Rightarrow$($\ell$-PEB):} See Corollary 1 of~\cite{ZDLZ21};

{\bf ($\ell$-PEB)$\Rightarrow$($\ell$-SEB):} By the definition of ($\ell$-PEB), for any $x\in[\bar{\phi}\leq\phi<\bar{\phi}+\nu]$, we have
\[
\begin{aligned}
\dist(x,[\phi\leq\bar{\phi}])\leq &c_1\|x-T_{\alpha}(x)\|\\
\leq &c_1\left(\frac{\alpha}{1-\alpha\rho}\right)\dist(0,\partial\phi(x))\qquad\mbox{(by statement (vi) of Proposition~\ref{prop})}.
\end{aligned}
\]
\item[{\rm(ii)}] {\bf ($\ell$-PEB)$\Rightarrow$($\ell$-PGAP):} For $x\in[\bar{\phi}\leq\phi<\bar{\phi}+\nu]$, let $x_p\in[\phi\leq\bar{\phi}]$ such that $\|x-x_p\|=\dist(x,[\phi\leq\bar{\phi}])$. Observe
\begin{equation}\label{eq:PEB-PGAP}
\begin{aligned}
&\phi(x)-\bar{\phi}
=\phi(x)-E_{\alpha}(x)+E_{\alpha}(x)-\bar{\phi}\\
\leq&\phi(x)-E_{\alpha}(x)+f(x)+\langle\nabla f(x),x_p-x\rangle+r(x_p)+\frac{1}{2\alpha}\dist^2(x,[\phi\leq\bar{\phi}])-f(x_p)-r(x_p)\\
=&\phi(x)-E_{\alpha}(x)+f(x)-f(x_p)+\langle\nabla f(x),x_p-x\rangle+\frac{1}{2\alpha}\dist^2(x,[\phi\leq\bar{\phi}]).
\end{aligned}
\end{equation}
By the $L$-gradient Lipschitz property of $f$,~\eqref{eq:PEB-PGAP} yields that
\[
\begin{aligned}
&\phi(x)-\bar{\phi}
\leq \phi(x)-E_{\alpha}(x)+\frac{1}{2}\left(\frac{1}{\alpha}+L\right)\dist^2(x,[\phi\leq\bar{\phi}])\\
\leq&\phi(x)-E_{\alpha}(x)+\frac{1}{2}\left(\frac{1}{\alpha}+L\right)c_1\|x-T_{\alpha}(x)\|^2\;\;\qquad\mbox{(by the $\ell$-PEB property)}\\
=&\alpha G_{\alpha}(x)+\frac{1}{2}\left(\frac{1}{\alpha}+L\right)c_1\|x-T_{\alpha}(x)\|^2\quad\qquad\mbox{(by statement (ii) of Proposition~\ref{prop})}\\
\leq&\alpha\left(1+\frac{(1+\alpha L)c_1}{1+\alpha\rho}\right)G_{\alpha}(x) \;\quad\quad\qquad\qquad \mbox{(by statement (iv) of Proposition~\ref{prop})}
\end{aligned}
\]
The desired result follows.

{\bf ($\ell$-PGAP)$\Rightarrow$($\ell$-KL):} By the definition of ($\ell$-PGAP), for any $x\in[\bar{\phi}\leq\phi<\bar{\phi}+\nu]$, we have
\[
\begin{aligned}
c_4[\phi(x)-\bar{\phi}]\leq&G_{\alpha}(x)\\
\leq&\frac{1}{2(1-\alpha\rho)}\dist^2(0,\partial\phi(x)) \quad\mbox{(by statement (v) of Proposition~\ref{prop})}
\end{aligned}
\]
\end{itemize}
\section{Proof of Proposition~\ref{prop:beb-2}} \label{app:rela_leb_2}
{\rm(i)} By the definition of ($\ell$-PEB) on $[\phi^*\leq\phi<\phi^*+\nu]$, for any $x\in[\phi^*\leq\phi<\phi^*+\nu]$, we have 
\[
\begin{aligned}
&\dist^2(x,[\phi\leq\phi^*])
\leq c_1^2\|x-T_{\alpha}(x)\|^2\\
\leq&\left(\frac{2c_1^2}{\frac{2}{\alpha}-L-\rho}\right)[\phi(x)-\phi(T_{\alpha}(x))]\quad\mbox{(by statement (iii) of Proposition~\ref{prop})}\\
=&\left(\frac{2c_1^2}{\frac{2}{\alpha}-L-\rho}\right)\left[\phi(x)-\phi^*\right]\qquad\quad\;\;\mbox{(since $\phi(T_{\alpha}(x))\geq\phi^*$)}
\end{aligned}
\]

{\rm(ii)} By the definition of (RSI), $\forall x\in[\phi^*\leq\phi<\phi^*+\nu]$ and $\xi\in\partial f(x)$, we have
\begin{equation}\label{eq:RSI_defi_1}
c_6\dist^2(x,[\phi\leq\phi^*])\leq\langle\xi,x-x_p\rangle\leq\|\xi\|\cdot\|x-x_p\|.
\end{equation}
Since $x_p=\arg\min_{y\in\setX_{\text{\tiny OP}}^*}\|x-y\|$ and $[\phi\leq\phi^*]=\setX_{\text{\tiny OP}}^*$, it follows from~\eqref{eq:RSI_defi_1} that
\[
c_6\dist(x,[\phi\leq\phi^*])\leq\|\xi\|,\quad\forall\xi\in\partial f(x),
\]
which yields the ($\ell$-SEB) on $[\phi^*\leq\phi<\phi^*+\nu]$. By statement (i) of Proposition~\ref{prop:beb}, the desired result follows.
\end{document}